\documentclass{article}

\usepackage{macros}

\title{Containing the spread of a contagion on a tree}

\author{
	Michela Meister \\ 
	meister@cs.cornell.edu
	\and
	Jon Kleinberg \\
	kleinberg@cornell.edu
	}

\begin{document}
\maketitle

\begin{abstract}
Contact tracing can be thought of as a race between two processes: an infection process and a tracing process. In this paper, we study a simple model of infection spreading on a tree, and a tracer who stabilizes one node at a time. We focus on the question, how should the tracer choose nodes to stabilize so as to prevent the infection from spreading further? We study simple policies, which prioritize nodes based on time, infectiousness, or probability of generating new contacts.
\end{abstract}

\section{Introduction} \label{section:introduction}
Mathematical models have played an important role in epidemiology, providing tools and frameworks complementing empirical and public health research. One key example are branching processes, which lead to the development of the $R_0$ metric for measuring the spread of disease~\cite{martcheva}. While there are many mathematical models of the spread of disease, far fewer models exist for contact tracing. Here we present an initial mathematical model of contact tracing which we use to explore algorithmic questions in designing contact tracing interventions.

Contact tracing is the iterative process of identifying individuals (the \textit{contacts}) exposed to an infected case~\cite{hiv_guide,tb_guide,vecino}. These contacts may then be tested for infection, treated, or quarantined, depending on the nature of the disease, to limit the spread of further infections. 

This can be thought of as a race between two processes: an infection process and a tracing process. The goal of the tracing process is to identify infected cases faster than the disease spreads, so that eventually no new infections occur, ie the infection is \textit{contained}. Contact tracing is often implemented by teams of human tracers, and as a result, the tracing process is limited by the number of human tracers available. A key strategic decision is how to maximize the effectiveness of this limited tracing capacity~\cite{muller-kretzschmar,kwok,spencer}. To simplify things, we model the tracing process as a single \textit{tracer} who is given a list of contacts exposed to infection. We focus on the question, given a list of contacts exposed to infection, which contact should the tracer investigate or \textit{query} next? In particular, how does the tracer's policy for querying contacts affect the probability that the infection is contained?

One of the challenges in studying these questions is the lack of simple models that manage to articulate trade-offs between the infection and tracing processes. In the contact tracing literature, there are few models which simultaneously capture the dynamics of an infection process and a resource-constrained tracing process. Recent surveys on the contact tracing literature specifically note that ``few models take the limited capacity of the public health system into account''~\cite{muller-kretzschmar} or ``consider...the practical constraints that resources for contact tracing and follow-up control measures might not be available at full throttle''~\cite{kwok}. Meanwhile, the literature on probabilistic models provides many epidemic models on trees, but these models do not consider the effect of a tracing process. Thus it seems as if a model describing the interaction between an infection process and a tracing process has been absent from these two fields. 

\paragraph{Related work.} A few other papers analyze contact tracing under resource constraints, however in somewhat different settings. In~\cite{meister-kleinberg}, Meister and Kleinberg develop a model of contact tracing in which the infection and tracing processes operate in two disjoint phases. In the first phase the infection spreads throughout the population; in the second phase the population is in ``lockdown'' and no new infections occur. Tracing proceeds in the second phase, and the tracer's objective is to identify infected nodes efficiently so as to maximize a total ``benefit''. In contrast, this paper studies concurrent infection and tracing processes, where the tracer's objective is to contain the spread of the infection.

Armbruster and Brandeau also study contact tracing under resource constraints in~\cite{armbruster-brandeau, armbruster-brandeau_sim}. In their model there are fixed resources to allocate across the two interventions, contact tracing and surveillance testing. The primary goal is to find the optimal allocation of resources so as to provide the best health outcomes for the population. They evaluate a few simple policies for prioritizing contacts in~\cite{armbruster-brandeau_sim}, and choose the policy that results in the lowest prevalence of infection for their main analysis in~\cite{armbruster-brandeau}. However, their main focus is on determining the best allocation of resources across these two systems. In comparison, our work focuses on analyzing and measuring the performance of different prioritization policies across a wide range of infection parameters. Finally, in a somewhat different context, Ben-Eliezer, Mossel, and Sudan study a mathematical model of information spreading on a network with errors in communication and investigate approaches for error correction~\cite{mossel}.

\paragraph{The model.} To address our questions, we need to be able to define trade-offs between the tracer's policy for querying individuals and factors such as an individual's rate of meeting new contacts and the probability that they transmit the infection to a contact. To do this, we develop a simple model of contact tracing on a tree, which involves concurrent infection and tracing processes. 

First we describe the infection process uninhibited by any tracing. Each individual is represented as a node with a binary infection status. A node $v$ is governed by two parameters: the probability $q_v$ that it meets a new contact and the probability $p_v$ that it transmits the infection to a contact. These parameters are sampled independently for each node, with $p_v \sim D_p$ and $q_v \sim D_q$. We will discuss more about these distributions later on, but the problem is still interesting even when both distributions take just a single fixed value. Initially all nodes are uninfected. In round $t = 0$ a node $r$ becomes infected with a probability drawn from $D_p$. In each round $t > 0$, each node $v$ meets a new contact $u$ with probability $q_v$. If $v$ is infected, it infects $u$ with probability $p_v$. This process generates a tree, where $r$ is the root, the nodes in the first layer are $r$'s contacts, the nodes in the second layer are contacts of those contacts, and so on. If a node $v$ joins the tree in round $t$, its \textit{time-of-arrival} is $\tau_v = t$. 

In order to define the tracing process, we assign each node a second binary status, indicating whether it is \textit{active} or \textit{stable}. A node that is \textit{active} probabilistically generates new contacts at each round, as defined by the infection process. A node that is \textit{stable} no longer generates new contacts and therefore cannot further spread the infection. Initially, every node is \textit{active}. 

Contact tracing starts once the infection is already underway, at time $t = k$, when the tracer identifies a root $r$ as an index case. From then on the tracer selects one node to \textit{query} at each step. Note that, while the tracer only queries one node at each step, we can change the rate of tracing relative to the infection process by changing the contact probability $q$. Increasing $q$ causes the infection to spread more quickly thereby decreasing the relative rate of tracing, and decreasing $q$ causes the infection to spread more slowly, thereby increasing the relative rate of tracing. Querying a node reveals its infection status, and if a node is infected, two events occur: (1) the node is stabilized and (2) the node's children (ie contacts) are revealed. Thus querying an infected node has two benefits: it prevents further infections and reveals individuals exposed to infection.\footnote{If an individual is found to be uninfected, they do not need to be stabilized, since they cannot infect any contacts.} At step $t = k$ the only node the tracer may query is the root $r$. From then on the tracer may only query a node if its parent is an infected node queried on an earlier step.

We now describe the concurrent infection and tracing processes. We say that an \textit{instance} of the contact tracing problem is defined by the three parameters $D_p$, $D_q$, and $k$. The process begins at step $t = 0$ when the root $r$ becomes infected with a probability drawn from $D_p$. The infection process runs uninhibited for steps $0 \leq t < k$. During each step $t \geq k$, first the tracer queries a node and then a single round of the infection process runs. During the infection round only active nodes generate new contacts. The infection is \textit{contained} if the tracer stabilizes all infected nodes. 

\paragraph{Policies for querying nodes.} We can think of the tracing process as maintaining a subtree where each node in the subtree has an infected parent. The \textit{frontier} is the set of all leaves in the subtree which have not yet been queried. We assume that the tracer observes the triple $(p_u, q_u, \tau_u)$ for each node in the subtree and that, for the purposes of querying, any two nodes in the frontier with the same triple of parameters are indistinguishable. A \textit{policy} is any rule that dictates which node from the frontier is queried next. Note that if the frontier is empty, then all infected nodes have been stabilized, and therefore the infection is contained. 

We say that a policy is \textit{non-trivial} if it only queries the children of infected nodes. (Since the children of uninfected nodes are guaranteed to be uninfected, there is no reason to query them.) The remainder of the paper considers only non-trivial policies.

\begin{figure}[htp]
\includegraphics[scale=.5]{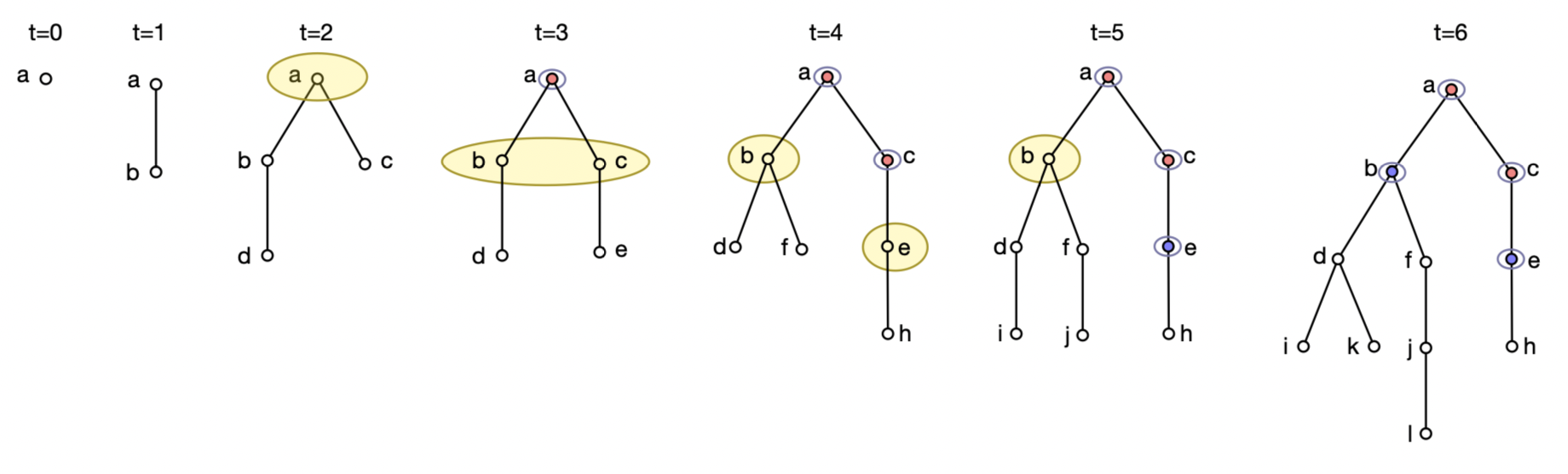}%
\caption{
This example illustrates concurrent infection and tracing processes. The infection process begins at $t=0$ and runs uninhibited for three steps. Tracing begins at $t=3$. At the start of step $t=3$, $a$ is the only node in the frontier. The tracer queries $a$, and since $a$ is infected, its children $b$ and $c$ join the frontier. Then another round of the infection process runs, in which every node that has not yet been queried probabilistically generates a new contact. In this case, $c$ generates child $e$. At $t=4$, the tracer queries $c$, which is infected, so $e$ joins the frontier. Another round of the infection process runs, where $b$ generates child $f$ and $e$ generates child $h$. At $t=6$, the tracer queries $b$, the sole node in the frontier. Node $b$ is uninfected, so the frontier is empty, and thus the infection is contained.}
\label{fig:intro_ex_descend}
\end{figure}

\subsection{Summary and overview of results} 
We analyze the effectiveness of different tracing policies, with a primary focus on the following question.

\begin{question}
How does the tracer's policy for querying nodes affect the probability that the infection is contained?
\end{question}

We study this question via both theoretical analysis and computational experiments. To begin, we establish basic theoretical bounds in~\cref{section:theory}, which characterize the performance of any non-trivial policy under certain conditions. In particular, we show that if either the infection probability $p_v$ or the contact probability $q_v$ is sufficiently small for all nodes, then any non-trivial policy contains the infection with high probability. On the other hand, we show that if both contagion parameters are sufficiently large, then every policy fails with high probability. 

Thus the results in this first section focus on settings in which policy choice is inconsequential; either any non-trivial policy is likely to contain the infection, or no non-trivial policy is likely to contain the infection. This motivates the question of whether there exists an instance in which the choice of policy \textit{is} significant.

\begin{restatable}{question}{mainques}
Is there an instance and a pair of policies $\A_1$ and $\A_2$ so that the probability of containment under $\A_1$ is greater than the probability of containment under $\A_2$?
\end{restatable}

This question is the primary focus of the remainder of the paper. To begin, we start with the simplest setting possible, where $D_p$ and $D_q$ are point mass distributions, that is, where all nodes have the same values of $p$ and $q$. In such a setting, one might think that the probability of containment ought to be agnostic to the policy chosen by the tracer. However, two nodes with the same infection and contact parameters $p$ and $q$ may still differ in their time-of-arrival $\tau$. 

There are two natural policies for ordering nodes by time-of-arrival. The \textit{ascending-time} policy orders nodes by ascending time-of-arrival, and the \textit{descending-time} policy orders nodes by descending time-of-arrival. In~\cref{section:separation} we prove that for a specific choice of $p$ and $q$, descending-time has a strictly higher probability of containment than ascending-time. Given this result, one might wonder whether descending-time is the better strategy in all instances. In~\cref{section:arrival} we compare the performance of ascending-time and descending-time for a large range of contagion parameters via computational experiments and find numerous instances in which we observe a significantly higher probability of containment for ascending-time. From these computational experiments, we find that each of the two policies commands a region of the parameter space of non-trivial size where it enjoys signficantly better performance, and we find that descending-time has the larger region.

This result still leaves open the question of how close these simple time-of-arrival policies are to an optimal policy. We study this question in~\cref{section:general} by training an optimal policy via reinforcement learning, where we formulate contact tracing as a game that the tracer ``wins'' if the infection is contained and ``loses'' otherwise. We train our policy via q-learning and find that its probability of containment is close to that of our policies based on time-of-arrival. This suggests that analyzing simple policies, such as prioritizing by time-of-arrival, gives us at least some insight into the performance of an optimal policy.

Since parameters $p$ and $q$ are fixed in the above settings, the only parameter with which to prioritize nodes is their time-of-arrival. In our final computational experiment, we explore settings in which each node $v$ has parameters $p_v$ and $q_v$ sampled from a distribution. Here we compare two other simple policies, prioritizing in descending order of $p_v$ or prioritizing in descending order of $q_v$, alongside the descending-time policy. We evalutate the performance of these policies across a range of distributions for $p_v$ and $q_v$ and find that the dominant policy is described by a phase diagram where the best policy in a specific setting depends on the distribution generating $p_v$ and $q_v$. 

\paragraph{Paper organization.} We begin with a summary of related work.~\Cref{section:theory} establishes basic theoretical bounds.~\Cref{section:separation} shows a specific instance where policy choice provably affects probability of containment.~\Cref{section:arrival} studies this question further, by comparing the performance of the ascending-time and descending-time policies via computational experiments.~\Cref{section:general} formulates contact tracing as reinforcement learning problem and conducts a heuristic search for other time-based polices to match the performance of ascending-time and descending-time.~\Cref{section:beyond} compares the performance of other policies for prioritizing nodes beyond time-based methods. Finally,~\cref{section:conclusion} presents future work and open questions.

\paragraph{Further Related Work}
Tian et al. study Tuberculosis contact tracing on a simulated network based on the population of Saskatchewan, Canada, and compare different prioritization policies for tracing individuals, with a particular focus on prioritizations based on patient demographics~\cite{tian}. Prior work by Fraser et al. and Klinkenberg studies when an outbreak of a disease may be contained by tracing and isolation interventions, with a focus on HIV, smallpox, and influenza, among other diseases~\cite{fraser,klinkenberg}. Hellewell et al. study this question for COVID-19 specifically~\cite{hellewell}. Kretzschmar et al. study the effect of time delays on contact tracing for COVID-19 via computational simulations~\cite{kretzschmar}. Kwok et al. review models of contact tracing and call for more models to account for resource constraints in tracing~\cite{kwok}. Kaplan et al. model a tracing and vaccination response to a bioterrorism attack in~\cite{kaplan_analyzing,kaplan_emergency}. 

Muller et al. study contact tracing as a branching process~\cite{muller}. Eames et al. study different contact tracing strategies for a compartmental model of infection~\cite{eames,hethcote-yorke}. Eames and Keeling study the relationship between the fraction of contacts which are traced and the rate at which the disease spreads in the context of sexually transmitted diseases~\cite{eames-keeling}.

\section{Basic Theoretical Bounds} \label{section:theory}
Recall that a non-trivial policy is one which only queries the children of infected nodes. Our basic theoretical bounds define conditions under which any non-trivial policy succeeds and under which any non-trivial policy fails. We focus on the following question.

\begin{question}
Fix a non-trivial policy $P$. Under what conditions, with high probability, does $P$ contain the infection? Under what conditions, with high probability, does $P$ fail to contain the infection?
\end{question}

\Cref{fig:section_2_thms} outlines our results in this section. First we establish that, if either the infection probability $p_v$ or contact probability $q_v$ is sufficiently small for all nodes, then any non-trivial policy contains the infection with high probability, which is shown in~\cref{thm:low_q,thm:low_p}. On the other hand, if both the infection probability and contact probability are sufficiently large for all nodes,~\cref{thm:runaway} shows that for any fixed non-trivial policy $P$, with high probablity $P$ does not contain the infection.

\begin{figure} [h]
\begin{center}
\includegraphics[scale=.5]{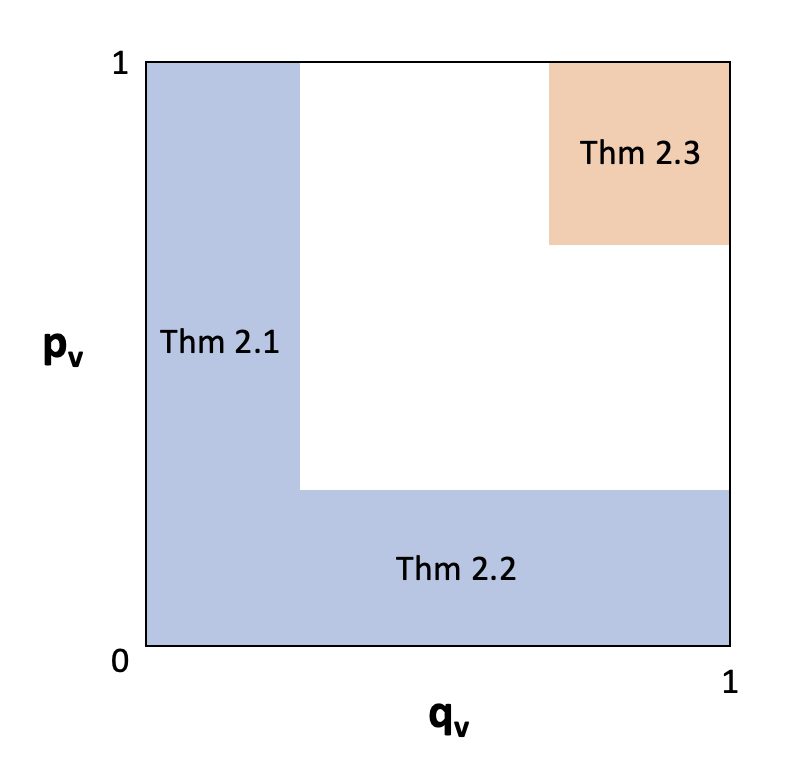}
\caption{Our basic theoretical bounds describe parameter regimes in which the choice of policy is inconsequential. If the infection or contact probability is very low for all nodes, then any non-trivial policy contains the infection with high probability, as shown in~\cref{thm:low_q,thm:low_p}. However, if the infection and contact probabilities are both very high for all nodes, then containment is highly unlikely, regardless of the non-trivial policy chosen, as shown in~\cref{thm:runaway}.}
\label{fig:section_2_thms}
\end{center}
\end{figure}

\subsection{Conditions under which containment is likely}
We present two theorems, which together show that if either the infection probability or contact probability is below a certain threshold for all nodes, then any non-trivial policy contains the infection with high probability. For the following two theorems it will be helpful to analyze the tracing process through the lens of deferred decisions. This analysis changes nothing about how the contact tracing process is defined, but simply makes it easier for us to analyze. First we will generate a \textit{transcript} $T_0', T_1', \dots$ of a tree with the given contagion parameters growing uninhibited by any tracing. During the tracing process, we construct infection tree $T_0, T_1, \dots$ by ``replaying'' the transcript. For example, when a new round of infection occurs at time $t$, we refer to $T_t'$ to determine the nodes to add to $T_t$ and their infection statuses. The benefit of this framework is that we can prove claims about the transcript $T_0', T_1', \dots$, which is often much easier to analyze, and show that these claims hold for the infection tree as well.

To start, we show that if the contact probability $q_v$ is sufficiently small for all nodes, then any non-trivial policy contains the infection with high probability. 
\begin{restatable}{theorem}{thmlowq}
\label{thm:low_q}
Fix a failure probability $\delta \in (0, 1)$ and an arrival time $k \in \NN$. Suppose that each node $v$ has infection probability $p_v \leq 1$. There is a $q(\delta, k) \in (0, 1]$ such that, if each node $v$ has contact probability $q_v < q(\delta, k)$, then any non-trival policy contains the infection with probability at least $1 - \delta$.
\end{restatable}

\Cref{appendix:low_q} provides the proof of~\cref{thm:low_q}. Similarly, the following theorem demonstrates that if for all nodes the probability of infection is below a certain threshold, then any non-trivial policy contains the infection with high probability.

\begin{restatable}{theorem}{thmlowp}
\label{thm:low_p}
Fix a failure probability $\delta \in (0, 1)$ and an arrival time $k \in \NN$. Suppose that each node $v$ has contact probability $q_v \leq 1$. There is a $p(\delta, k) \in (0, 1]$ such that, if each node $v$ has infection probability $p_v < p(\delta, k)$, then any non-trival policy contains the infection with probability at least $1 - \delta$.
\end{restatable}
\Cref{appendix:low_p} provides the proof of~\cref{thm:low_p}.

With the above two theorems, we have established that if either the infection probability or the contact probability is below a certain threshold, any policy contains the infection with high probability.

\subsection{Conditions under which containment is unlikely}
Here we show that if both the infection probability and contact probability are above a certain threshold for all nodes, then the infection is unlikely to ever be contained.
\begin{theorem} \label{thm:runaway}
Fix a policy $P$. Fix $\delta \in (0,1)$ and $k \geq 3$. There exist $p, q < 1$ such that, if for all nodes $v$ $p_v \geq p$ and $q_v \geq q$, with probability at least $1 - \delta$, $P$ does not contain the infection.
\end{theorem}
\begin{proof}
Let $h = 2 \lceil \max(128, \ln(4/\delta) / 2) \rceil + 16$. Let $f(\delta) = \max((1 - \delta/2)^{1/h}, 1/2)$, and observe that $f(\delta) < 1$. Choose $p < 1$ and $q < 1$ such that $pq \geq f(\delta)$. By~\cref{lem:f_delta} with probability at least $1 - \delta$, policy $P$ does not contain the infection.  
\end{proof}
The following lemma supports~\cref{thm:runaway}. The proof of~\cref{lem:f_delta} is deferred to~\cref{appendix:f_delta}, along with proofs for additional supporting lemmas. 
\begin{restatable}{lemma}{lemfdelta}
\label{lem:f_delta}
Fix a policy $P$. Fix $\delta \in (0, 1)$ and $k \geq 3$. Suppose that there are probabilities $p, q \in (0, 1)$ such that for all nodes $v$ $p_v \geq p$ and $q_v \geq q$. There exists a function $f(\delta) < 1$ such that if $pq > f(\delta)$, then with probability at least $1-\delta$, $P$ does not contain the infection. Specifically, setting $h= 2 \lceil \max(128, \ln(4/\delta) / 2) \rceil + 16$, we can choose $f(\delta) = \max((1 - \delta/2)^{1/h}, 1/2)$. 
\end{restatable}
The idea for the proof of~\cref{lem:f_delta} has two parts. First we show that, if at any point in time, there are at least $B$ active infected nodes, then it is unlikely the infection will ever be contained. Second, we show that, with high probability, there is a time $t$ with at least $B$ active infections. We take a union bound over these two events to prove the theorem.

Thus, we've shown that there are settings in which any non-trivial policy is likely to contain the infection, as well as settings in which \textit{no} non-trivial policy is likely to contain the infection. However, does it matter which non-trivial policy we employ? The remainder of the paper studies the following question in a variety of different settings.
\mainques*
The following section proves that yes, there are settings in which different policies result in different probabilities of containment.

\section{Policy Choice Affects Probability of Containment} \label{section:separation}
In this section we show that policy choice affects the probability of containment by providing an instance in which two non-trivial policies have different probabilities of containment. Recall the two policies presented in~\cref{section:introduction}, ascending-time and descending-time. Ascending-time prioritizes nodes in order of increasing time-of-arrival and descending-time prioritizes nodes in order of decreasing time-of-arrival. In this section we demonstrate an instance in which descending-time has a strictly higher probability of containment than ascending-time. 

\begin{figure}[h]
\subfloat[Ascending-time queries the node with the earliest time-of-arrival.]{%
  \includegraphics[scale=.5]{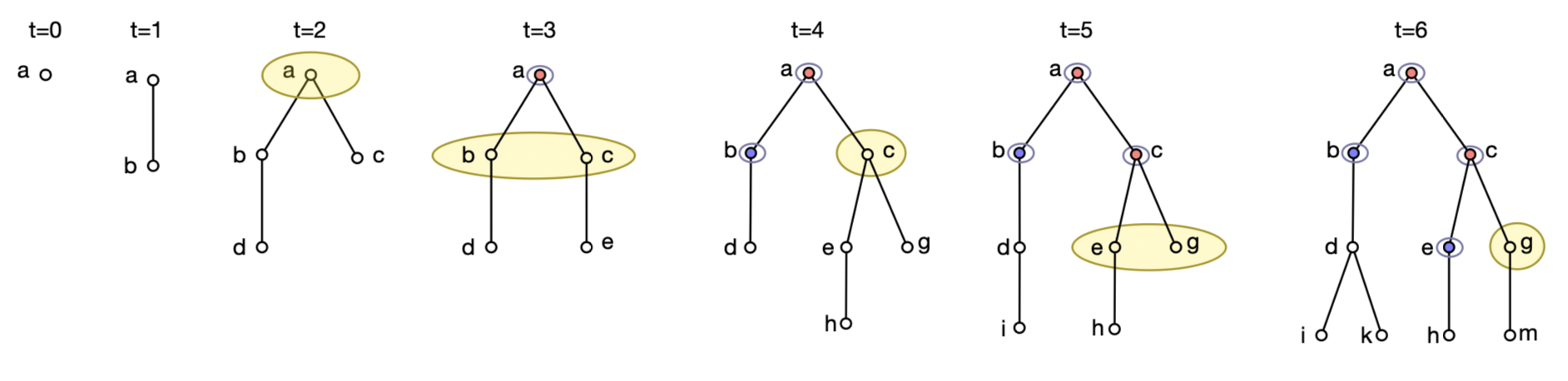}%
}
\vfill
\subfloat[Descending-time queries the node with the latest time-of-arrival.]{%
  \includegraphics[scale=.5]{figures/descend_ex}%
}
\caption{We analyze two policies which prioritize nodes based on time-of-arrival, ascending-time and descending-time. Ascending-time queries the node in the frontier with the earliest time-of-arrival, while descending-time queries the node in the frontier with the latest arrival time. The example illustrates ascending-time and descending-time operating on the same transcript. The infection process begins at $t=0$. At the start of $t=3$, node $a$ is the only node in the frontier. At $t=3$, both policies query $a$. Since $a$ is infected, its children $b$ and $c$ join the frontier. Node $b$ has time-of-arrival $\tau_b = 1$ while node $c$ has time-of-arrival $\tau_c = 2$, so at $t=4$ ascending time queries $b$ while descending-time queries $c$. By the end of $t=6$, ascending-time still has a non-empty frontier, while descending-time has contained the infection. In~\cref{thm:separation} we prove that, for a certain setting of infection parameters, descending-time has a strictly higher probability of containment than ascending-time.}
\label{fig:ascend_v_descend}
\end{figure}

\begin{figure}[h]
\centering
\subfloat[At $t=4$, ascending-time queries $b$.]{%
  \includegraphics[scale=.5]{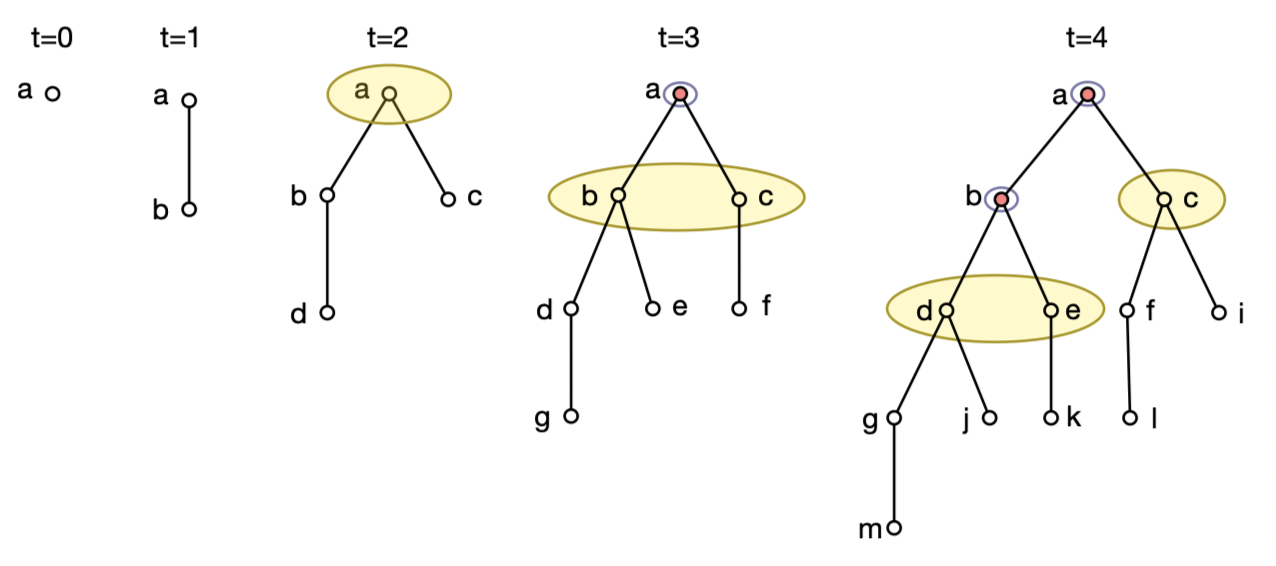}%
}
\vfill
\subfloat[At $t=4$, descending-time queries $c$.]{%
  \includegraphics[scale=.5]{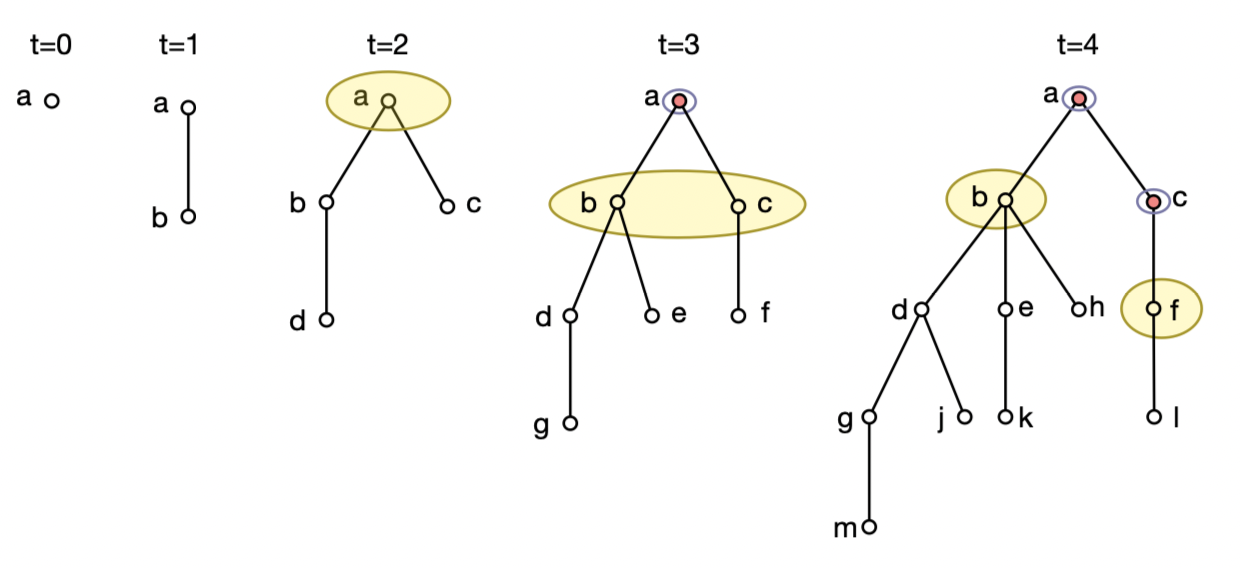}%
}
\caption{The example shows ascending-time and descending-time operating on the same transcript in a setting where all nodes have the same infection probability $p=.9999985$ and contact probability $q=1$. The infection process begins at $t=0$ and runs uninhibited for three steps. At $t=3$, both policies query $a$, and since $a$ is infected its children $b$ and $c$ both join the frontier. Node $b$ has time-of-arrival $\tau_b = 1$ while node $c$ has time-of-arrival $\tau_c = 2$. Therefore at $t=4$ ascending-time queries $b$ while descending-time queries $c$. We analyze these two choices in~\cref{thm:separation} to prove that, for this setting of paramters $p$ and $q$, descending-time has a strictly higher probability of containment.}
\label{fig:separation_pf}
\end{figure}

\Cref{fig:ascend_v_descend} provides an example that illustrates how the two policies ascending-time and descending-time operate. In~\cref{thm:separation} we prove that, for a certain setting of infection parameters, descending-time has a strictly higher probability of containment than ascending-time. Our analysis in this proof is centered around the processes illustrated in~\cref{fig:separation_pf}.


\begin{theorem} \label{thm:separation}
There is an instance in which two policies have different probabilities of containment.
\end{theorem}
\begin{proof}
We provide an instance in which ascending-time and descending-time have different probabilities of containment. Let $p = 0.9999985$, let $q = 1$, and let $k = 3$. Consider the instance where for all nodes $v$, $p_v = p$ and $q_v = q$. For this instance, let $\P_A$ be the probability of containment for ascending-time, and let $\P_D$ be the probability of containment for descending-time. By analyzing the first few nodes each policy queries, we will show that $\P_D > \P_A$.

First let us review the infection process up until tracing begins at time $t = k$, as illustrated in~\cref{fig:separation_pf}. Note that, since $q = 1$, until tracing begins every node generates a new child at every step. At $t = 0$ the root $a$ is infected with probability $p$. At $t = 1$ $a$ generates child $b$. If $a$ is infected, $a$ infects $b$ with probability $p$. At $t = 2$ $a$ generates child $c$ and $b$ generates child $d$. If $a$ is infected, $a$ infects $b$ with probability $p$, and if $b$ is infected, $b$ infects $d$ with probability $p$. Let a node $v$ have time-of-arrival $\tau_v$. Therefore, $\tau_b = 1$, $\tau_c = 2$, and $\tau_d = 2$.

The proof analyzes five different infection status outcomes for the nodes $a$, $b$, and $c$, which partition the space of all outcomes, as summarized in~\cref{fig:tbl}. We define a failure parameter $\delta = .001$. For a node $v$, $E_v$ is the event that $v$ is infected. For each outcome $u$, we bound the probability of containment for both policies conditional on outcome $u$, where $\P_A(u)$ is the probability of containment for ascending-time and $\P_D(u)$ is the probability of containment for descending-time. We then upper bound $\P_A$ and lower bound $\P_D$ by computing the average probability of containment over the outcomes $u$, weighted by the probability $\Pr(u)$ that the outcome $u$ occurs.

The first outcome is $u_1 = \lnot E_a$, which occurs with probability $1 - p$. Since the root $a$ is not infected, the infection is always contained, so $\P_A(u_1) = \P_D(u_1) = 1$. For all the following outcomes, the root $a$ is infected, so $E_b$ and $E_c$ are independent, and each occurs with probability $p$. For the second outcome, $u_2 = E_a \land \lnot E_b \land \lnot E_c$, neither $b$ nor $c$ are infected, so the infection is always contained, and therefore $\P_A(u_2) = \P_D(u_2) = 1$ as well.~\Cref{lem:u_3,lem:u_4,lem:u_5} bound the probabilities of containment for the remaining three outcomes.

\begin{figure} [h]
\begin{center}
\begin{tabular}{ ||c|c|c|c|| } 
 \hline
$u$ & $\Pr(u)$ & $\P_A(u)$ & $\P_D(u)$ \\
 \hline
 \hline
 $\lnot E_a$ & $1 - p$ & $1$ & $1$ \\ 
 \hline
 $E_a \land \lnot E_b \land \lnot E_c$ & $p(1-p)^2$ & 1 & 1 \\
 \hline
 $E_a \land \lnot E_b \land E_c$ & $p^2 (1 - p)$ & $\leq \delta$ & 1 \\
 \hline
 $E_a \land E_b \land \lnot E_c$ & $p^2 (1 - p)$ & $\leq \delta$ & $\leq \delta$ \\
 \hline
 $E_a \land E_b \land E_c$ & $p^3$ & $\leq \delta^2$ & $\leq \delta$ \\
 \hline
\end{tabular}
\caption{ By considering the infection status outcomes for the first few nodes in the tree, we can bound the probabilities of containment for ascending-time and descending-time. Here $u$ is an outcome describing the infection status of nodes $a, b, c$ and $\Pr(u)$ is the probability that outcome $u$ occurs. Given the outcome $u$, $\P_A(u)$ is the probability that ascending-time contains the infection, and $\P_D(u)$ is the probability that descending-time contains the infection.}
\label{fig:tbl}
\end{center}
\end{figure}

Using the bounds in~\cref{fig:tbl}, we can lower bound $\P_D$ and upper bound $\P_A$. 
\begin{align*}
\P_D &\geq (1 - p) + p(1 - p)^2 + p^2 (1 - p) \\
\P_A &\leq (1 - p) + p(1 - p)^2 + 2 \delta p^2 (1 - p) + \delta^2 p^3
\end{align*}

Therefore,
\begin{align*}
\P_D - \P_A &\geq p^2 (1 - p) - 2 \delta p^2 (1 - p) - \delta^2 p^3 \\
&> 4.97 \times 10^{-7}
\end{align*}
Thus there is an instance for which descending-time has a strictly higher probability of containment than ascending-time.
\end{proof}

The remainder of this section provides analysis for the third, fourth, and fifth outcomes in~\cref{fig:tbl}. To help with this process, we first review the tracing process for both policies, which is illustrated in~\cref{fig:separation_pf}. Recall that for all $t \geq 3$, first the tracer queries a node from the frontier, and then a round of the infection process runs. For a node $v$, let $A(v)$ be the time at which ascending-time queries $v$, and let $D(v)$ be the time at which descending-time queries $v$. At $t = 3$, the root $a$ is the only node available to query. Therefore, both ascending-time and descending-time query $a$ at $t = 3$, so $A(a) = D(a) = 3$. If $a$ is infected, at time $t = 4$ both policies have frontier $\{ b, c \}$. Since $\tau_b < \tau_c$, ascending-time queries $b$ and descending-time queries $c$, so $A(b) = 4$ and $D(c) = 4$. The following lemmas analyze the implications of this difference between the two policices.

\begin{lemma} \label{lem:u_3}
For the outcome $u_3 = E_a \land \lnot E_b \land E_c$, $\P_A(u_3) \leq \delta$ and $\P_D(u_3) = 1$.
\end{lemma}
\begin{proof}
Since $b$ is not infected, the frontier for ascending-time at time $t = 5$ is $\{ c \}$. Therefore, ascending-time queries $c$ at time $t = 5$, so $A(c) = 5$. Since $a$ is infected, $c$ is infected with probability $p$, and since $\tau_c = 2$, $c$ has been present for $3$ timesteps already. Therefore $c$ is the root of a subtree with the same infection parameters $p$ and $q$, and with tracing delay $k = 3$. Thus, by~\cref{lem:check_pq}, ascending-time contains the infection within the subtree rooted at $c$ with probability at most $\delta$, so $\P_A(u_3) \leq \delta$.

Now consider the tracing process for descending-time. Descending-time queries $c$ at time $t = 4$, at which point $c$ has exactly one child, which is a leaf. Since $c$ is infected, the child is added to the frontier. Then one round of the infection process runs, and the child of $c$ generates a leaf of its own. At time $t = 5$, the frontier for descending-time contains exactly $b$ and the child of $c$. Since the child of $c$ by definition has a larger arrival time than $c$, and since $c$ has a larger arrival time than $b$, at time $t = 5$ descending-time queries the child of $c$. If the child of $c$ is infected, its child (which now has a single leaf of its own) is added to the frontier, and the process repeats. Thus, descending-time recursively queries the descendents of $c$, which create a long chain, until it reaches a descendent which is not infected, which since $p < 1$, occurs with probability $1$. Therefore, with probability $1$, time descending stabilizes all infected nodes within the subtree rooted at $c$, so since $b$ is not infected, $\P_D(u_3) = 1$.
\end{proof}

\begin{lemma} \label{lem:u_4}
For the outcome $u_4 = E_a \land E_b \land \lnot E_c$, $\P_A(u_4) \leq \delta$ and $\P_D(u_4) \leq \delta$.
\end{lemma}
\begin{proof}
First let us consider the tracing process for ascending-time. Since $b$ is infected, at the start of $t = 5$ the frontier for ascending-time is $\{ c, d, e \}$. Therefore, ascending-time queries $d$ at some time $t \geq 5$, so $A(d) \geq 5$. Since $b$ is infected, $d$ is infected with probability $p$, and since $\tau_d = 2$, $d$ has been present for $3$ timesteps already. Therefore, by the same argument as before, by~\cref{lem:check_pq} ascending-time contains the infection within the subtree rooted at $d$ with probability at most $\delta$, so $\P_A(u_4) \leq \delta$. 

Now consider the tracing process for descending-time. Since $c$ is not infected, at time $t = 5$ the frontier for descending-time is $\{ b \}$. Therefore descending-time queries $b$ at time $t = 5$, so $D(b) = 5$. Since $b$ has an infected parent, and has been present for at least $3$ timesteps, by~\cref{lem:check_pq} descending-time contains the infection within the subtree rooted at $b$ with probability at most $\delta$, so $\P_D(u_4) \leq \delta$.
\end{proof}

\begin{lemma} \label{lem:u_5}
For the outcome $u_5 = E_a \land E_b \land E_c$, $\P_A(u_5) \leq \delta^2$ and $\P_D(u_5) \leq \delta$.
\end{lemma}
\begin{proof}
First consider the tracing process for ascending-time. Since $b$ is infected, at the start of $t = 5$ the frontier for ascending-time is $\{ c, d, e \}$. Therefore, $A(c) \geq 5$ and $A(d) \geq 5$. Both $c$ and $d$ have infected parents, so their subtrees of infection are independent. Therefore, by~\cref{lem:check_pq} ascending-time contains the infection within the subtree rooted at $c$ with probability at most $\delta$, and independently, contains the infection within the subtree rooted at $d$ with probability at most $\delta$. Thus $\P_A(u_5) \leq \delta^2$. 

Now consider the tracing process for descending-time. Recall that at the start of $t=4$ the frontier for descending-time is exactly $\{ b, c \}$, and that descending-time queries $c$ at time $t = 4$. Therefore $D(b) \geq 5$. By the same argument as given in~\cref{lem:u_3}, descending-time contains the infection within the subtree rooted at $c$ with probability $1$. However, since $D(b) \geq 5$ and $\tau_b = 2$, by~\cref{lem:check_pq} descending-time contains the infection within the subtree rooted at $b$ with probability at most $\delta$, so $\P_D(u_5) \leq \delta$.
\end{proof}

Finally, we show the following supporting lemma which proves that the chosen infection and contagion parameters are large enough such that neither descending-time nor ascending-time are likely to contain the infection if a node has already been active for a few steps.
\begin{lemma} \label{lem:check_pq}
Fix a policy $P$. Let $\delta = .001$. Let $p = 0.9999985$, $q = 1$, and suppose $k \geq 3$. Suppose each node $v$ has infection probability $p$ and contact probability $q$. Then with probability at least $1 - \delta$, policy $P$ does not contain the infection.
\end{lemma}
\begin{proof}
Following~\cref{lem:f_delta}, we simply need to check that $pq > f(\delta)$.
\begin{align*}
f(\delta) &= \max((1 - \delta/2)^{1/(2 \lceil \max(128, \ln(4/\delta) / 2) \rceil + 16))}, 1/2) \\
&= \max((1 - .001/2)^{1/272}, 1/2) \\
&= (1 - .0005)^{1/272} \\
&\leq 0.9999982 \\
&< 0.9999985 \\
&= pq
\end{align*}
Therefore, any policy contains the infection with probability at most $\delta$.
\end{proof}

\section{Time-of-Arrival Heuristics} \label{section:arrival}
As introduced in the previous section, there are two obvious policies for prioritizing nodes by time-of-arrival. The \textit{descending-time} policy prioritizes nodes in order of descending time-of-arrival, and the \textit{ascending-time} policy prioritizes nodes based on ascending time-of-arrival. The main question is, which policy has a higher probability of containment? 

The previous section provided a single setting of contagion parameters $p$ and $q$ in which descending-time provably has a higher probability of containment than ascending-time. In this section we explore the performance of both policies across a wide range of contagion parameters via computational experiments. Our main finding is that each of the two policies commands a substantial region of the parameter space in which it enjoys a higher probability of containment than the other. These computational experiments extend the conclusions of the previous section by demonstrating numerous instances in which policy choice affects the probability of containment. Moreover, since neither policy has the higher probability of containment in all instances, any comparison between these two policies must take into account the contagion parameters. 

Finally, our results qualitatively suggest that a trade-off between $p$ and $q$ may define the boundaries of these regions of dominance. Further characterizing these regions is an intriguing direction for future work.

\paragraph{Simulation overview.} 
Our computational experiments focus on the simple setting where every node is governed by the same infection parameters $p, q \in [0, 1]$. As a result, nodes differ only by time-of-arrival. Our simulation implements the model described in~\cref{section:introduction}, where $k = 3$, $D_p$ is the constant distribution on $p$, and $D_q$ is the constant distribution on $q$. Therefore an instance is defined by the pair $(p, q)$. As defined in the model from~\cref{section:introduction}, the infection is not contained if an infinite number of nodes become infected. Since checking this condition is intractable, for the purposes of our simulation we redefine containment in terms of a constant $Z_C = 10$; the infection is not contained if more than $Z_C$ nodes are active and infected. If the tracer stabilizes every infected node before this threshold is reached, then the infection is contained. 

Even with such a constraint, the infection tree may grow quite large before either of the above two terminating conditions is reached. To manage the size of the tree, our implementation only includes nodes in the tree that could at some point enter the frontier, which are exactly the infected nodes and their children. Even so, there are still instances in which this tree could grow very large.\footnote{An interested reader can return to the proof of~\cref{thm:separation} and the corresponding illustration in~\cref{fig:separation_pf} for such an example. Suppose the tracer follows the descending-time policy. After stabilizing $A$, the tracer stabilizes $C$ and then explores the subtree rooted at $C$. As explained in the proof, the descending-time policy constrains the subtree of $C$ to a single chain of nodes, which extends by one node at each step. The tracer stabilizes infected nodes along the chain until they reach the first uninfected node. As is also explained in the proof, at any step during this process, only the last two nodes in the chain are active. Now suppose that $B$ is uninfected. Then any active infections are in the chain rooted at $C$, of which there are at most two, so the threshold $Z_C$ is never reached. In this case, tracing only terminates once the tracer finds the first uninfected node in the chain, thus stabilizing all infected nodes. At this point the chain has expected length $1/(1-p)$, so for $p$ close to 1 this results in a very large tree.} Therefore, we limit the size of the tree to $Z_T = 1000$ nodes. We will argue below that this has only negligible effects on the computational results.

A \textit{trial} is a single run of the simulation and terminates in one of the following three states:
\begin{enumerate}[label=(\roman*)]
	\item \textit{The infection is contained: } All infected nodes are stabilized.
	\item \textit{The infection is not contained: } The number of active infections exceeds $Z_C$.
	\item \textit{The trial did not converge: } The number of nodes in the tree exceeds $Z_T$ before either $(i)$ or $(ii)$ occur.
\end{enumerate}

Nearly all trials terminate in states $(i)$ or $(ii)$; out of the approximately $2.88 \times 10^{11}$ trials run in total for all the computational experiments we present, only $87$ trials terminated in state $(iii)$.

For a fixed instance $(p, q)$, we define a policy's probability of containment to be the probability that a trial terminates in state $(i)$. A policy's \textit{observed} probability of containment for a given instance over a series of trials is the fraction of trials which terminate in state $(i)$. Our goal in the computational experiments that follow is to determine, for a given instance, which policy has the higher probability of containment based on each policy's observed probability of containment over a series of trials.

\paragraph{Choosing parameters $Z_C$ and $Z_T$.} To choose the parameter $Z_C$, we ran $2.5 \times 10^5$ trials for each of the descending-time and ascending-time policies in the setting where $D_p = \Unif(0, .75)$ and $D_q$ is the constant distribution on $1$. These trials were run with a slightly different set-up: instead of tracing beginning in round $k=3$, each trial was initialized with two roots and tracing began immediately thereafter. We found that every trial which reached $10$ active infections also reached $50$ active infections. We set $Z_C = 10$.

We set $Z_T = 1000$. Given the miniscule fraction of trials that terminated in state $(iii)$, it seems unlikely that increasing $Z_T$ would affect our results.

\begin{figure*}
        \centering
        \begin{subfigure}[b]{0.4\textwidth}
            \centering
            \includegraphics[width=\textwidth]{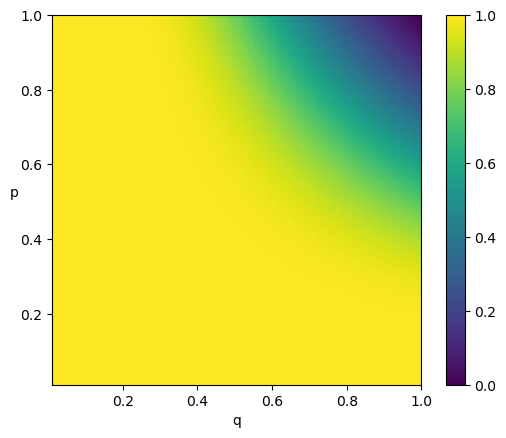}
            \caption[]%
            {{\small Observed probability of containment: ascending-time}}    
            \label{fig:time_ascend_plot}
        \end{subfigure}
        \hfill
        \begin{subfigure}[b]{0.4\textwidth}  
            \centering 
            \includegraphics[width=\textwidth]{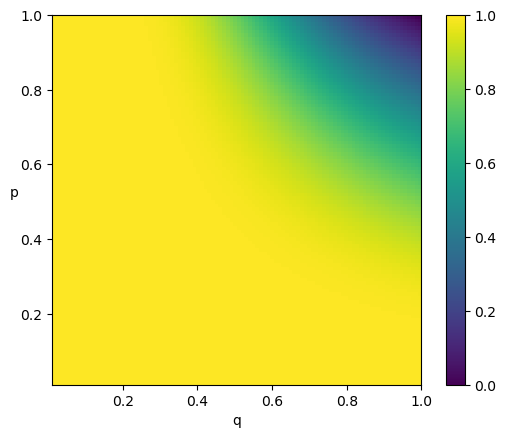}
            \caption[]%
            {{\small Observed probability of containment: descending-time}}    
            \label{fig:time_descend_plot}
        \end{subfigure}
        \\
        \begin{subfigure}[b]{0.4\textwidth}   
            \centering 
            \includegraphics[width=\textwidth]{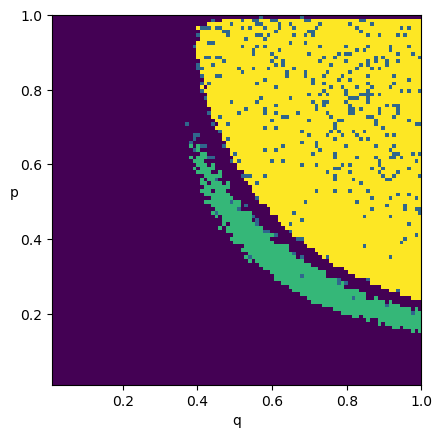}
            \caption[]%
            {{\small Regions of dominance}}    
            \label{fig:diff_asc_plot}
        \end{subfigure}
        \caption[]%
        {\small Plots (a) and (b) show the observed probability of containment for ascending-time and descending-time, respectively, from our first round of computational experiments. Both policies contain the infection with high probability when at least one of $p$ or $q$ is low. When $p$ and $q$ are both large, the observed probability of containment decreases. Plot (c) shows regions where the two policies have significantly different probabilities of containment. In the yellow region,  descending-time is significantly better, and in the green region time ascending-time is significantly better. The purple and blue points indicate instances where we do not have sufficient confidence to state which policy is better.} 
        \label{fig:arrival}
    \end{figure*}

\paragraph{Computational experiments, part 1.} Our computational experiments compare the probabilities of containment for descending-time and ascending-time as the infection probability $p$ and the contact probability $q$ range from $.01$ to $1$ in increments of $.01$. Since there are $100$ possible values each of $p$ and $q$, this makes for $10^4$ instances in total. For each instance $(p, q)$, we run $N = 7.5 \times 10^6$ trials for each policy. The observed probabilities of containment across the entire parameter space are displayed in plots (a) and (b) in~\cref{fig:arrival}.

\paragraph{Results, part 1.}
Plots in (a) and (b) in~\cref{fig:arrival} summarize the results of our experiments. Intuitively, as $p$ and $q$ increase, the observed probability of containment decreases for both policies. For example, when at least one of $p$ or $q$ is at most $0.4$, the smallest observed probability of containment for ascending-time is $0.875$ and for descending-time is $0.902$. The observed probability of containment drops off quickly when both $p$ and $q$ are large. When $p=q=.9$, the observed probability of containment for ascending-time is $0.231$ and for descending-time is $0.293$; when $p=q=.95$, the observed probability of containment for descending-time is $0.108$ and for descending-time is $0.148$. 

In comparing the two policies, we find that as $p$ and $q$ both grow large, descending-time has the higher observed probability of containment in many instances, and often by large margins. Even in instances where ascending-time exhibits the higher observed probability of containment, the margins are too small to make any claims of statistical significance. It is therefore consistent with this first round of computational experiments that descending-time might always be the better policy. Could there be any instance where ascending-time has the higher probability of containment?

\paragraph{Computational experiments, part 2.} 
To investigate this further, we run a second round of computational experiments, in which we focus solely on instances where the absolute difference between the observed probabilities of containment for the two policies is above a fixed threshold. For these instances, we run a second round of trials. We choose the number of trials to run for a given instance so that, if the difference in the observed probabilities of containment from the first round is approximately maintained over the second round, we will have run enough trials to make claims of statistical significance. As a result, the number of trials run in the second round for a given instance is a function of the absolute difference between the observed probabilities for the two policies in the first round. 

Based solely on the trials run in the second round, we compute the observed probabilities of containment for the two policies. We say that a policy \textit{dominates} an instance if we determine with confidence at least $1/2$ that it has a higher probability of containment than the other. Instances where ascending-time dominates are colored green, and instances where descending-time dominates are colored yellow. All other instances are colored purple.

To describe the process for the second round of trials more formally, for an instance $(p, q)$ let $\mathcal{P}_A(p,q)$ and $\mathcal{P}_D(p,q)$ be the probabilities of containment for ascending-time and descending-time, respectively. Let $\mathcal{P}_A'(p,q)$ and $\mathcal{P}_D'(p,q)$ be the observed probabilities of containment for ascending-time and descending-time, respectively, from the first round of trials. Let $d(p, q) = \lvert \mathcal{P}'_A(p, q) - \mathcal{P}'_D(p, q) \rvert$. If $d(p, q) \geq .00035$, we run a second set of $M(d(p, q))$ trials, where we chose $M(d(p, q)) = 50 \cdot \left\lceil \left\lceil 3 \ln(1 / .15) /  (.49 d(p, q))^2 \right\rceil / 50 \right\rceil$.\footnote{There were $4031$ instances where $d(p,q) \geq .00035$, which warranted a second round of trials. Restricted to these instances, the median number of trials run in the second round was $153,250$, the maximum was $193,503,050$, and the minimum was $1500$. In total, across all instances, $95,933,663,250$ trials were run in the second round.}

Let $\mathcal{P}''_A(p, q)$ and $\mathcal{P}''_D(p, q)$ be the observed probabilities of containment for ascending-time and descending-time, respectively, computed solely from the second round of $M(d(p,q))$ trials. Via~\cref{lem:2coin}, we compute the confidence that $\max\{ \mathcal{P}_A(p, q), \mathcal{P}_D(p, q) \} = \max \{ \mathcal{P}''_A(p, q), \mathcal{P}''_D(p, q) \}$. If the confidence is at least $1/2$, the instance $(p, q)$ is colored green if the maximizing policy is ascending-time and yellow if the maximizing policy is descending-time.

Finally, to explore whether it is possible to achieve higher confidence bounds, we run a series of $1.5 \times 10^9$ trials for each policy for instance parameters $p=.19$ and $q=1$, since this emerged from the first set of computational experiments as a natural candidate instance where ascending-time might be the better policy.

\paragraph{Results, part 2.} 
Plot (c) in~\cref{fig:arrival} summarizes the results of the second round of trials. As shown in the plot, descending-time dominates a substantial region of the parameter space where $p$ and $q$ are both large, while ascending-time dominates a band directly below. There are $534$ instances where ascending-time dominates and $3129$ instances where descending-time dominates. We make no claims for the remaining $6538$ instances. As a note, one can qualitatively observe these regions of dominance after running far fewer trials, however a large number of trials is necessary for our confidence guarantees.

Given the confidence bounds for plot (c), we can make some further conclusions. First, we know that in any $5 \times 5$ square of green instances, with confidence at least $1 - 2^{-25}$, there is at least one instance where $\mathcal{P}_A(p,q) > \mathcal{P}_D(p,q)$. Likewise, and with the same confidence, in any $5 \times 5$ square of yellow instances there is at least one instance where $\mathcal{P}_D(p,q) > \mathcal{P}_A(p,q)$. Therefore, with high probability there is a yellow cresent of instances where descending-time has the higher probability of containment, as well as a green crescent of instances directly below where ascending-time has the higher probability of containment. We discuss some compelling directions for future work related to characterizing this border region in~\cref{section:conclusion}.

Finally, while the guarantees for a individual instance in plot (c) only hold with confidence $1/2$, higher confidence guarantees in these regions are also possible. From our series of $1.5 \times 10^9$ trials for instance parameters $p=.19$ and $q=1$ we found that $\mathcal{P}_A(.19, 1) > \mathcal{P}_D(.19, 1)$ with confidence at least $1 - 10^{-10}$.
	
\section{General Policies} \label{section:general}
There are numerous ways of prioritizing nodes based on time-of-arrival. So far we have focused on only two policies, ascending-time and descending-time. We call these policies \textit{monotonic}, because they prioritize nodes in order of increasing or decreasing time-of-arrival, respectively. Could there be some other way of prioritizing nodes by arrival time that has a much higher probability of containment for certain instances?

To better understand the optimality of ascending-time and descending-time, we compare their performance to other policies that prioritize nodes based on time-of-arrival. Since there are countless ways to prioritize nodes based on time-of-arrival, we use an imperfect, heuristic search to identify competitive policies for comparison. Our search finds no other policy which far outperforms both monotonic policies. To be clear, this null result does not preclude the existence of such policies, which perhaps could be identified via more advanced techniques. Rather, the fact that monotonic policies perform on par with policies found via a broad, heuristic search, suggests that ascending-time and descending-time are reasonable policies to analyze and study. 

As a first step towards implementing our heuristic search, we formulate contact tracing as a partially-observable Markov decision process. Then, for each instance we implement a heuristic version of Q-learning to train an instance-specific policy. We then compare the resulting policy's observed probability of containment over a series of trials to that of ascending-time and descending-time. 

\paragraph{Partially-observable Markov decision processes.}
A Markov decision process (MDP) involves an agent that probabilistically moves through a set of states $S$ by choosing different actions. The agent begins at time $t = 0$ in an initial state $s_0 \in S$. At time $t \geq 0$, the agent chooses an action $a_t$ from the set $A_{s_t}$ of actions available from state $s_t$. The agent then receives reward $r(s_t, a_t)$ and transitions to a state $s_{t+1} \sim P(s_t, a_t)$, where $P(s_t, a_t)$ is a distribution on $S$. Conditioned on $s_t$ and $a_t$, the reward $r(s_t, a_t)$ and next state $s_{t+1}$ are independent of all prior states $s_0, \dots, s_{t-1}$ and actions $a_0, \dots, a_{t-1}$~\cite{mohri}.

A partially-observable Markov decision process (POMDP) is an MDP in which the agent does not observe the current state. Instead, after taking action $a_t$ and transitioning to state $s_{t+1}$ the agent makes an observation $o_{t+1} \sim O(s_{t+1}, a_t)$, where $O(s_{t+1}, a_t)$ is a distribution on the set $\Omega$ of all observations the agent can make~\cite{kaelbling}.

There is a wide literature on both of these processes, which is far too vast to be summarized here. A deeper overview of both processes is given in~\cite{littman,puterman}.

\paragraph{Formulating contact tracing as a POMDP.}
At time $t$, the state $s_t$ consists of the infection tree $T_t$ and the time $t$, denoted $s_t = (T_t, t)$. The tree $T_t$ is defined as in~\cref{section:introduction}, where each node is labeled with the following dimensions: its infection status, whether it is active or stable, and its parameters $p$ and $q$. Recall that the frontier $f_{s_t}$ is exactly the set of active children of infected nodes queried on earlier steps. Taking an action corresponds to querying a node from the frontier, so $f_{s_t}$ defines the set of actions available to the tracer from $s_t$. 

Suppose the tracer queries a node $v_t \in f_{s_t}$. To generates state $s_{t+1}$, first $v_t$ is stabilized and its infection status is revealed, and then one round of the infection process runs, as in~\cref{section:introduction}. The resulting tree is $T_{t + 1}$, and $s_{t + 1} = (T_{t + 1}, t + 1)$. The entire process continues from state $s_{t + 1}$ and the tracer receives reward $r(s_t, v_t) = 0$, unless one of the following two conditions is met. If the frontier $f_{s_{t+1}}$ is emtpy, the tracer receives reward $r(s_t, v_t) = 1$ and the process terminates. If the number of active infections in $T_{t+1}$ exceeds $Z_C$, the tracer receives reward $r(s_t, a_t) = -1$ and the process terminates.

In state $s_t$, the tracer's observation $o(s_t)$ consists of the subtree of nodes in $T_t$ which the tracer has already queried and the frontier $f_{s_t}$. Specifically, the tracer observes all dimensions of nodes queried on prior steps and observes all dimensions \textit{except} infection status for nodes in the frontier. Note that, as a results, $o(s_t)$ is a deterministic function of $s_t$.

\paragraph{Q-learning heuristic search.} 
Recall that our goal is to compare ascending-time and descending-time to other competitive time-based policies. Given that there are countless ways to prioritize nodes based on time-of-arrival, we need some way to search over time-based policies. For this, we turn to Q-learning, an algorithm that iteratively learns actions which induce large rewards in different states. In this way, Q-learning provides a heuristic search for competitive time-based policies. 

We now describe our Q-learning implementation more formally. To simplify our implementation, instead of working with the POMDP framework detailed above, we work with \textit{partial states} $s = (f, t)$, where $f_s$ is the frontier and $t_s$ is the current time. Let $V$ be a table of values where each row corresponds to a partial state $s = (f_s, t_s)$ and each column corresponds to an arrival time $\tau$. In this setting, all nodes have the same infection parameters $p$ and $q$, so $f_s$ is simply a multiset of arrival times. Since choosing an action corresponds to querying a node from the frontier, and since nodes are only distinguishable to the tracer by their arrival times, the set of arrival times in $f_s$ are exactly the actions available from $s$. 

Fix parameters $\eps, \alpha, \gamma \in [0, 1]$. To select an action from $s$,
\begin{itemize}
	\item With probability $\eps$, select an action $a$ uniformly at random from $f_s$. 
	\item Otherwise, select $\arg\max_{a \in f_s} V(s, a)$.
\end{itemize}
Suppose that after taking action $a$ from $s$, the process transitions to $s'$, and the tracer receives reward $r(s')$. The table $V$ is then updated to reflect the value of taking action $a$ from partial state $s$.
\begin{align*}
	V(s, a) \leftarrow (1 - \alpha) \cdot V(s, a) + \alpha  (r(s') + \gamma \cdot \max_{a \in f_{s'}} V(s', a) )
\end{align*}
Once training has concluded, the resulting table $V$ can be translated into the following straightforward policy: in partial state $s$ select $\arg\max_{a \in f_s} V(s, a)$.~\cite{watkins,kansal}

A challenge of training policies via the above approach is that the space of partial states quickly grows very large, and the table $V$ needs to keep track of all possible partial states. To make training tractable, we restrict the length of a training episode so that $t \leq 4$, we limit the size of the frontier to at most $3$ nodes, and we restrict the frontier to nodes with arrival time at most $3$. If an episode reaches one of these limits, the reward returned is based on the current frontier $f_s$. Specifically, we return the average reward of $100$ trials, each initialized with the frontier $f_s$, for a fixed \textit{terminal policy}, which is either descending-time or ascending-time. While an episode may terminate in one of these two policies, choices at the beginning of the episode remain unconstrained. On of our goals is to understand whether this extra flexibility at the beginning of an episode results in larger probabilities of containment. 

\paragraph{Computational experiments.} We repeat the experiments from~\cref{section:arrival} with policies trained via the Q-learning process described above and compare the observed probabilities of containment of our learned policies to those of ascending-time and descending-time. As in~\cref{section:arrival}, $k = 3$ and in a given instance every node has the same contagion parameters $p$ and $q$. Parameters $p$ and $q$ each range from $.01$ to $1$ in increments of $.01$, which results in a total of $10^4$ instances. For each instance $(p, q)$, we train two Q-learning policies, one with ascending-time as the terminal policy, which we call \textit{learn-ascend}, and one with descending-time as the terminal policy, which we call \textit{learn-descend}. We run $10^6$ episodes to train each policy, with parameters $\eps=0.1$, $\alpha=0.1$, and $\gamma=0.6$. We then evaluate each policy on $10^6$ trials. 

In order to compare these evaluations to our previous results, we need to convert reward averages to probabilities of containment. Since a reward of $1$ means the infection was contained, and a reward of $-1$ means the infection was not contained, an average reward of $R$ corresponds to a probability of containment of $(R + 1)/2$.

\paragraph{Results.} Recall that we call ascending-time and descending-time monotonic, because they order nodes in either increasing or decreasing time-of-arrival, and our goal is to determine whether there is some other policy which far outperforms these two monotonic policies in a given instance. For each instance $(p,q)$, we compute the difference between the maximum observed probability of containment for our two monotonic policies, ascending-time and descending-time, and the maximum observed probability of containment for our two heuristic policies, learn-descend and learn-ascend. We also compute the difference relative to the maximum observed probability of containment for the two monotonic policies. In both cases, a positive difference indicates that a monotonic policy is better and a negative difference indicates that a heuristic policy is better.  

In our search, no heuristic policy far outperformed the monotonic policies on any instance. The minimum difference was $-0.00228$ and the minimum relative difference was $-0.00712$. Of course, one may wonder whether our search was thorough enough. While we only performed a heuristic search, we think it was a reasonable search in that there are few instances in which the heuristic policies have far worse performance than the monotonic policies. For example, the median difference is $10^{-4}$, the mean relative difference is $0.00023$. The difference is at most $0.005$ in all but $80$ instances, and the relative difference is at most $0.01$ in all but $59$ instances. The maximum difference is $0.07$ and the maximum relative difference is $.3143$.

Thus these results demonstrate that monotonic policies perform on par with policies allowed the freedom to choose any short starting sequence of queries. While these results do not at all rule out the existence of other more competitive policies, they do suggest that ascending-time and descending-time are reasonable policies to analyze and study.

\section{Beyond Time-of-Arrival Heuristics} \label{section:beyond}
In all the instances we have seen so far, every node is governed by the same contagion parameters $p$ and $q$. In this section we explore instances where different nodes may have different contagion parameters. This opens the door to many more potential policies. Whereas nodes in previous sections were only distinguishable to the tracer based on their time-of-arrival, now a node's contagion parameters may play a factor in its prioritization. The main question is, how should we prioritize nodes, given these three dimensions $p_v$, $q_v$, and $t_v$?

We explore three simple policies for prioritizing nodes: by descreasing infection probability $p_v$, by decreasing contact probability $q_v$, and by decreasing time-of-arrival $t_v$, which is exactly the descending-time policy. We compare the performance of these three policies across a range of instances via computational experiments. We say that a policy \textit{dominates} an instance if, of the three policies, it has the highest probability of containment. We find that each policy dominates a large region of the instance space, which further supports our conclusions from~\cref{section:arrival} that the performance of different policies varies across instances. 

Of course, each of the policies we consider prioritizes nodes based on only a single dimension, and there are clearly numerous other policies to consider. A compelling direction for future work is to analyze policies which prioritize nodes based on multiple dimensions. 

\paragraph{Computational experiments.}
We compare the performance of the three policies across a range of instances. Recall that we are studying a mathematical process with a genuine probability of containment which is different from the average observed over a series of trials. The goal of our computational experiments is to determine which policy has the highest probability of containment in each instance, based on each policy's observed performance over a series of trials. A node $v$ has infection probability $p_v \sim D_q$, where $D_q$ is the uniform distribution on $[p_{\min}, 1)$, and contact probability $q_v \sim D_q$, where $D_q$ is the uniform distribution on $[q_{\min}, 1)$. Otherwise, the trials follow the same process as in~\cref{section:arrival}. To generate our space of instances, we let $p_{\min}$ and $q_{\min}$ range independently from $0$ to $1$ in increments of $.01$. This results in a total of $10,201$ instances.

For each instance, we run $5 \times 10^5$ trials for each policy. Via~\cref{lem:3coin_bd}, we compute the confidence that the policy with the highest observed probability of containment is also the policy with the highest (genuine) probability of containment. If the confidence is at least $1/2$, the instance is marked with the color of the dominating policy. 

\paragraph{Results.}
\Cref{fig:slate_plot} describes our results. Recall that these computational experiments involve infection and contact parameters drawn from uniform distributions on $[p_{\min}, 1)$ and $[q_{\min}, 1)$, respectively, while in~\cref{section:arrival} all nodes in a given instance are governed by the same parameters $p$ and $q$. First, we see that when the infection and contact probabilities for all nodes are large, which happens when $p_{\min}$ and $q_{\min}$ are both large, descending-time dominates. 

We do not yet have a mathematical understanding of these regions or an explanation for when one policy might dominate another, however we do draw one insight from these results. Consider the green region where $p_{\min}$ is small and $q_{\min}$ is large. Since $D_q = \Unif[q_{\min}, 1)$, a large value of $q_{\min}$ implies that $D_q$ has small variance. Meanwhile, since $D_p = \Unif[p_{\min}, 1)$, a small value of $p_{\min}$ implies that $D_p$ has large variance. As we see in the plot, in this region the policy which prioritizes by probability of infection $p_v$ dominates. A similar phenomenon occurs in the yellow region where $q_{\min}$ is small and $p_{\min}$ is large. In this region the policy which prioritizes by contact probability $q_v$ dominates. The intuition that we take away from this plot is that if there are very large differences in the variances of $D_p$ and $D_q$, it is better to prioritize by the parameter with the larger variance. There are numerous directions for future work in this areas, which we outline in~\cref{section:conclusion}.
 
\begin{figure} 
\begin{center}
\includegraphics[scale=.5]{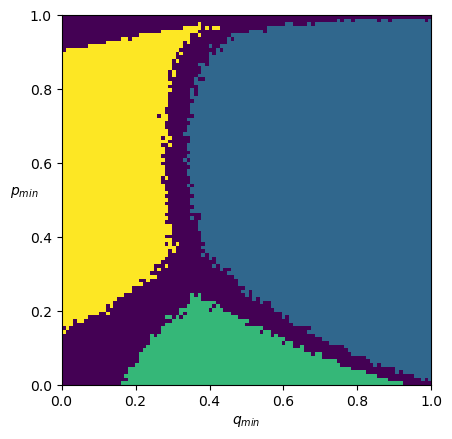}
\caption{ The above figure illustrates results from computational experiments where $p_v \sim \Unif[p_{\min}, 1)$ and $q_v \sim \Unif[q_{\min}, 1)$. Each pixel corresponds to an instance defined by $(p_{\min}, q_{\min})$.  We compare the probabilities of containment for three policies: prioritizing by $p_v$, prioritizing by $q_v$, and descending-time. We say that a policy \textit{dominates} if its probability of containment is signficantly higher than that of the other two policies. Each pixel is colored to indicate which of the three policies, prioritizing by $p_v$ (green), prioritizing by $q_v$ (yellow), or descending-time (blue) dominates. Purple indicates areas of no confidence. }
\label{fig:slate_plot}
\end{center}
\end{figure}
\section{Conclusion} \label{section:conclusion}
In this paper, we model contact tracing as a race between two processes: an infection process and a tracing process. The tracing process is implemented by a tracer, who queries one node each step, with the goal of containing the infection. The main question we study is, which policy for querying nodes should the tracer follow so as to have a high probability of containing the infection? We show, via both theoretical bounds and computational experiments, that policy choice affects probability of containment in many different settings and that the best policy to deploy may depend on the parameters of the contagion.

There are numerous compelling directions for future work going forward. First, while we have evidence from computational experiments in~\cref{section:arrival} that there are settings where ascending-time outperforms descending-time, we do not have a proof for such a setting. An interesting direction for future work would be to prove that there is an instance where ascending-time has the higher probability of containment, which would complement the result for descending-time in~\cref{section:separation}. Second, computational experiments in~\cref{section:arrival} qualitatively show that there is a curve in the $p-q$ parameter space, where descending-time dominates above the curve and ascending-time dominates a band below the curve. Further characterizing this curve and border region is an intriguing direction for future study. Finally, it is interesting to consider and analyze other prioritization policies. For example, in~\cref{section:beyond} the three policies we consider each prioritize nodes across only one dimension. It would be exciting to understand how to prioritize nodes based on multiple dimensions. For example, how should we prioritize two nodes with the same time-of-arrival, but where one has the higher infection probability and the other has the higher contact probability? More broadly, developing theoretical methods for comparing policies based on all three parameters seems like a first step toward exploring optimal policies in these settings.

\section*{Acknowledgments}
Michela Meister was supported by the Department of Defense (DoD) through the National Defense Science \& Engineering Graduate (NDSEG) Fellowship Program. Jon Kleinberg was supported in part by a Simons Investigator Award, a Vannevar Bush Faculty Fellowship, MURI grant W911NF-19-0217, AFOSR grant FA9550-19-1-0183, ARO grant W911NF19-1-0057, and a grant from the MacArthur Foundation.

\bibliographystyle{plain}
\bibliography{references}

\section{Deferred Proofs from~\cref{section:theory}}

\subsection{Proving~\cref{thm:low_q}} \label{appendix:low_q}
\thmlowq*

\begin{proof}
Recall that in each step, either the frontier is empty and therefore the infection is contained, or the tracer stabilizes one node from the frontier. Therefore, at any step $t > k$, either the infection is already contained, or the tracer has stabilized at least $t - k$ nodes. 

Let $T_0', T_1', \dots$ be the transcript of a tree with the given contagion parameters. The proof idea is to show that there is a time $m > k$ at which point with probability at least $1 - \delta$, $T_m'$ has at most $m - k$ nodes. By definition, $T_m$ has at most the number of nodes in $T_m'$. Therefore, with probability at least $1 - \delta$, the tracer has contained the infection by time $t = m$. 

Set $m = \lceil e/\delta \rceil + k$ and $q(\delta, k) = 1 / (\lceil e/\delta \rceil + k)$. Let $X_t'$ be the number of nodes in $T_t'$. Observe that $X_0' = 1$. For all $n > 0$, $\Exp[X_n'] \leq X'_{n - 1}(1 + q)$. Therefore $\Exp[X_n'] \leq (1 + q)^n$. 

By~\cref{lem:apply_markov}, with probability at least $1 - \delta$, $X_m' \leq m - k$. Since $T_m$ has at most as many nodes as $T_m'$, with probability at least $1 - \delta$ the infection is contained by time $t = m$. 
\end{proof}

\begin{lemma} \label{lem:apply_markov}
Fix $\delta \in (0, 1)$. Fix $k \in \NN$. Let $m = \lceil e/\delta \rceil + k$ and let $u = 1/m$. Let $Z$ be a random variable such that $\Exp[Z] \leq (1 + u)^m$. Then with probability at least $1 - \delta$, $Z \leq m - k$. 
\end{lemma}
\begin{proof}
By Markov's inequality
\begin{align*}
\Pr \left( Z \geq m - k \right) \leq \frac{\Exp[Z]}{m - k} \leq \frac{(1 + u)^m}{m - k}
\end{align*}

Noting that $u = 1/m$ we have that
\begin{align*}
\frac{(1 + u)^m}{m - k} \leq \frac{e^{um}}{m - k} \leq \frac{e}{m - k} \leq \frac{e}{\lceil e/\delta \rceil + k - k} \leq \frac{e}{\lceil e/\delta \rceil } \leq \frac{e}{ e/\delta } \leq \delta
\end{align*}

Thus $\Pr \left( Z \geq m - k \right) \leq \frac{\Exp[Z]}{m - k} \leq \delta$.
\end{proof}

\subsection{Proving~\cref{thm:low_p}} \label{appendix:low_p}
\thmlowp*
\begin{proof}
Recall that the infection is contained when every infected node is stabilized. Also, recall that only nodes with infected parents ever enter the frontier. Therefore, if the frontier is empty, then every node with an infected parent is stabilized, and thus the infection is contained. Additionally, recall that in each step, either the frontier is empty, or the tracer stabilizes one node from the frontier. Therefore, by time $t > k$, either the infection is contained of the tracer has stabilized at least $t - k$ nodes with infected parents. 

This proof follows a similar structure as the proof for the preceding theorem, with one slight modification. In this case, we analyze a super transcript $T_0'', T_1'', \dots$ with the given infection parameters $p_v$ but where each node has contact probability $1$. Creating the transcript $T_0', T_1', \dots$ from this super transcript simply requires selecting each out-edge of a node $v$ with probability $q_v$. Thus, at any time $t$, $T_t$ has at most as many nodes with infected parents as $T_t''$.

The proof idea is to bound the number of nodes with infected parents in the transcript $T_0'', T_1'', \dots$. We show that there is a time $m > k$ such that with probability at least $1 - \delta$, $T''_m$ has fewer than $m - k$ nodes with infected parents. $T_m$ has at most as many nodes with infected parents as $T_m'$, so with probability at least $1 - \delta$ the infection is contained by step $m$.

Set $m = \lceil e/\delta \rceil + k$ and $p(\delta, k) = 1 / (\lceil e/\delta \rceil + k)$. Let $Y''_t$ be the number of nodes which have an infected parent in $T_t''$. Let $Y_{t,d}''$ be the number of nodes which have an infected parent at depth $d$ in $T_t''$. Observe that, since every node in the super transcript has contact probability $1$, $T''_t$ is a binomial tree. Therefore, there are $n_d = \binom{t}{d}$ nodes in the $d$-th level of $T''_t$. The root by definition has an infected parent, so $Y''_{t,0} = 1$. A node in level $d \geq 1$ is infected if all of its ancestors are infected, which occurs with probability at most $p^d$. Thus, for any $1 \leq d \leq t$,
$$\Exp[Y''_{t,d}] \leq n_d p^d = \binom{t}{d} p^d $$ 
Therefore,
$$\Exp[Y''_t] \leq \sum_{d = 0}^t \Exp[Y''_{t, d}] \leq \sum_{d = 0}^t \binom{t}{d} p^d \leq (1 + p)^t$$

By~\cref{lem:apply_markov}, with probability at least $1 - \delta$, $Y_m'' \leq m - k$. Since $T_m$ has at most as many nodes with infected parents as $T_m''$, with probability at least $1 - \delta$ the infection is contained by time $t = m$. 
\end{proof}

\subsection{Proving~\cref{lem:f_delta}} \label{appendix:f_delta}
\lemfdelta*
\begin{proof}
Set $p, q < 1$ such that $pq > f(\delta)$.\footnote{ Since $0 < \delta < 1$, for any $r > 0$, $(1 - \delta/2)^{1/r} < 1$, so $f(\delta) < 1$, and therefore this setting of $p$ and $q$ is possible.} Set $B = 2 \lceil \max(128, \ln(4/\delta) / 2) \rceil$. Let $X_t$ be the number of active infections at time $t$. By~\cref{lem:reach_B}, with probability at least $1 - \delta/2$, there is a time $t$ where $X_t \geq B$. 

Let $C = 2 \lceil \max(64/(pq), \ln(4/\delta) / 2) \rceil$. Since $pq > f(\delta) > 1/2$, $B \geq C$. By~\cref{lem:B_onward}, given that there is a time $t$ where $X_t \geq C$, with probability at least $1 - \delta/2$, policy $P$ does not contain the infection. A union bound over these two events proves the claim.
\end{proof}

The following lemma shows that with high probability, there is a time $t$ with at least $B$ active infections.
\begin{lemma} \label{lem:reach_B}
Fix $\delta \in (0, 1)$ and $k \geq 3$. Set $B = 2 \lceil \max(128, \ln(4/\delta) / 2) \rceil$. Let $f(\delta) = \max((1 - \delta/2)^{1/(2 \lceil \max(128, \ln(4/\delta) / 2) \rceil + 16))}, 1/2)$. Set $p,q \in (0, 1)$ such that $pq > f(\delta)$. Suppose that for all nodes $v$ $p_v \geq p$ and $q_v \geq q$. Let $X_t$ be the number of active infections at time $t$. With probability at least $1 - \delta/2$, there is a time $t$ where $X_t \geq B$. 
\end{lemma}
\begin{proof}
Without loss of generality, let $k = 3$.\footnote{Any $k > 3$ corresponds to a larger head start for the infection, so without loss of generality it suffices to analyze the case where the infection has the smallest head start.} Let $t = \lceil \log(B) \rceil + k$. Suppose that the root is infected, and that every round, every active, infected node generates an infected child. Then by~\cref{lem:bound_tracing}, $X_t \geq B$. 

We now bound the probability that in each round, every active infected node generates an infected child. There are $t$ rounds of infection, so at most $2^t$ infected nodes can potentially be generated, where
\begin{align*}
2^t = 2^{\lceil \log(B) \rceil + k} \leq 2^{\log(B) + k + 1} = 2^{\log(B) + 4} = B + 16
\end{align*}
In a given round, an infected node generates an infected child with probability at least $pq$. Therefore, the probability that all $B + 16$ nodes exist and are infected is at least $(pq)^{B + 16}$. Since $pq > f(\delta)$, 
\begin{align*}
(pq)^{B + 16} &\geq (f(\delta))^{B + 16} \\
&\geq \left( \max((1 - \delta/2)^{1/(2 \lceil \max(128, \ln(4/\delta) / 2) \rceil + 16))}, 1/2) \right)^{B + 16} \\
&= \left( \max((1 - \delta/2)^{1/(B + 16))}, 1/2) \right)^{B + 16} \\
&\geq \left( (1 - \delta/2)^{1/(B + 16)} \right)^{B + 16} \\
&= 1 - \delta/2
\end{align*}
\end{proof}

Given that there are at least $B$ active infections, the following lemma shows that then it is unlikely the infection is ever contained.
\begin{lemma} \label{lem:B_onward}
Fix a policy $P$. Fix $\delta \in (0, 1)$. Suppose that there are probabilities $p, q \in (0, 1)$ such that for all nodes $v$ $p_v \geq p$ and $q_v \geq q$. Let $X_t$ be the number of active infections at time $t$. Let $C = 2 \lceil \max(64/(pq), \ln(4/\delta) / 2) \rceil$. If there is a time $t$ where $X_t \geq C$, then with probability at least $1 - \delta/2$, $P$ does not contain the infection.
\end{lemma}
\begin{proof}
Call $C/2$ rounds of the tracing and infection process an $\textit{epoch}$. Starting from step $t$, enumerate the epochs $m = 1, 2, \dots$. By assumption, there is a time $t$ where $X_t \geq C$. By~\cref{lem:induction}, for any $M \in \NN$, with probability at least $1 - \sum_{m=1}^M \exp(-C2^{m})$, at the end of the $M$-th epoch there are at least $2^MC$ active infections. Since
$$\sum_{m=1}^{\infty} \exp(-C 2^{m+1}) \leq 2 \exp(-2C) \leq 2 \exp(-2 \ln(4 / \delta)/2) \leq \delta/2,$$ 
with probability at least $1 - \delta/2$ policy $P$ does not contain the infection.
\end{proof}

A key to the above proof is that, once the number of active infections reaches the threshold $B$, it continues to grow exponentially, as shown in the following lemma. 
\begin{lemma} \label{lem:induction}
Fix a policy $P$. Fix $\delta \in (0, 1)$. Suppose that there are probabilities $p, q \in (0, 1)$ such that for all nodes $v$ $p_v \geq p$ and $q_v \geq q$.
Let $C = 2 \lceil \max(64/(pq), \ln(4/\delta) / 2) \rceil$, and suppose there is a time $t$ where $X_t \geq C$. Call $C/2$ rounds of the tracing and infection process an $\textit{epoch}$. Starting from step $t$, enumerate the epochs $m = 1, 2, 3, \dots$. For any $M \in \NN$, with probability at least $1 - \sum_{m=1}^M \exp(-C2^{m})$, at the end of the $M$-th epoch there are at least $2^MC$ active infections.
\end{lemma}
\begin{proof}
Consider the base case, $M = 1$. By assumption, at the start of the first epoch there are at least $C$ active infections. Since an epoch has $C/2$ rounds, the tracer stabilizes at most $C/2$ active infected nodes during the first epoch. Therefore, there are at least $C/2$ active infected nodes which are not stabilized during the first epoch, and whose descendents are also not stabilized during the first epoch. The proof idea is to lower bound the number of infections these nodes generate over the course of the epoch. Enumerate the these nodes $i = 1, 2, \dots, C/2$. Let $p_i$ and $q_i$ be the infection and contact probabilities for node $i$, respectively. Enumerate the rounds of the epoch $j = 1, \dots, C/2$. 

Let $Y_{i,j}$ indicate the event that node $i$ generates an infected child in round $j$. Then $Y_{i,j} \sim \Ber(p_iq_i)$. Then the number of active infected nodes generated in the first epoch is at least
$$Y = \sum_{j=1}^{C/2} \sum_{i=1}^{C/2} Y_{i,j}$$ 
Let
$$\mu = \Exp[Y] = \sum_{j=1}^{C/2} \sum_{i=1}^{C/2} \Exp[Y_{i,j}] = \frac{C}{2} \sum_{i=1}^{C/2} p_iq_i.$$
Define $\mu' = (C/2)^2 pq$. Since for all $i$, $p_iq_i \geq pq$, $\mu \geq \mu'$. 
Then
\begin{align}
\Pr(Y \leq 2C) &\leq \Pr(Y \leq C^2(pq)/8) \label{eqn:base_1} \\
&= \Pr(Y \leq \mu' / 2) \label{eqn:base_2} \\
&\leq \Pr(Y \leq \mu/2) \label{eqn:base_3} \\
&\leq \exp(-\mu/8) \label{eqn:base_4} \\
&\leq \exp(-\mu'/8) \label{eqn:base_5} \\
&= \exp(-C^2 pq / 32) \label{eqn:base_6} \\
&\leq \exp(-2C) \label{eqn:base_7} 
\end{align}
To understand~\cref{eqn:base_1}, observe that $pq \geq f(\delta) \geq 1/2$, and that $C \geq 128$. Therefore $C^2(pq)/8 \geq C^2/16 \geq 2C$.~\Cref{eqn:base_2} applies the definition of $\mu'$.~\Cref{eqn:base_3} applies $\mu \geq \mu'$,~\cref{eqn:base_4} is due to~\cref{chernoff},~\cref{eqn:base_5} applies $\mu \geq \mu'$,~\cref{eqn:base_6} applies the definition of $\mu'$, and~\cref{eqn:base_7} uses the fact that $C \geq 64/(pq)$. Thus at the end of the first epoch with probability at least $1 - \exp(-2C)$, there are at least $2C$ active infections. 

Assume that for any $l \geq 1$, at the end of the $l$-th epoch with probability at least $1 - \sum_{m=1}^l \exp(-C 2^{m+1})$ there are at least $2^lC$ active infections. Now consider $M = l+1$. With probability at least $1 - \sum_{m=1}^l \exp(-C 2^{m+1})$ we begin with at least $2^lC$ active infections. Since $l \geq 1$, and since an epoch contains only $C/2$ rounds, at least $2^lC/2$ active infected nodes will not be traced during the $l+1$-th epoch. Thus at least $2^lC/2$ active infected nodes probabilistically generate an infected child in each round of the epoch. 

Following the same notation as the base case, enumerate the nodes $i = 1, 2, \dots, 2^lC/2$, and enumerate the rounds $j = 1, \dots, C/2$. Let $Y_{i,j} \sim \Ber(p_iq_i)$ indicate the event that the node $i$ generates an infected child in round $j$. Then the number of active infected nodes generated in epoch $l+1$ is at least
\begin{align*}
Y = \sum_{j=1}^{C/2} \sum_{i=1}^{2^lC/2} Y_{i,j} 
\end{align*}
Let  
\begin{align*}
\mu = \Exp[Y] = \sum_{j=1}^{C/2} \sum_{i=1}^{2^lC/2} \Exp[Y_{i,j}] = \sum_{j=1}^{C/2} \sum_{i=1}^{2^lC/2} p_iq_i
\end{align*}
Define
\begin{align*}
\mu' &= 2^l C^2 pq / 4
\end{align*}
Note that $\mu \geq \mu'$. Then
\begin{align}
\Pr(Y \leq 2^{l+1}C) &\leq \Pr (Y \leq 2^l C^2 pq / 8) \label{eqn:ind_1}\\
&\leq \Pr(Y \leq \mu'/2) \label{eqn:ind_2} \\
&\leq \Pr(Y \leq \mu/2) \label{eqn:ind_3} \\
&\leq \exp(-\mu/8) \label{eqn:ind_4} \\
&\leq \exp(-\mu'/8) \label{eqn:ind_5} \\
&= \exp(-2^l C^2 pq / 32) \label{eqn:ind_6} \\
&\leq \exp(-2^{l+1} C) \label{eqn:ind_7}
\end{align}
To understand~\cref{eqn:ind_1}, observe that $pq \geq f(\delta) \geq 1/2$, and that $C \geq 128$. Therefore $2^l C^2 pq / 8 \geq 2^l C^2 / 16 \geq 2^{l+1} C$.~\Cref{eqn:ind_2} applies the definition of $\mu'$.~\Cref{eqn:ind_3} applies $\mu \geq \mu'$,~\cref{eqn:ind_4} is due to~\cref{chernoff},~\cref{eqn:ind_5} applies $\mu \geq \mu'$,~\cref{eqn:ind_6} applies the defintion of $\mu'$, and~\cref{eqn:ind_7} uses the fact that $C \geq 64/(pq)$. 

Thus, given that there are at least $2^lC$ active infections at the start of epoch $l + 1$, with probability at least $1 - \exp(-2^{l+1} C)$ there are at least $2^{l+1}C$ active infected at the end of epoch $l + 1$. A union bound shows that with probability at least $1 - \sum_{m=1}^{l+1} \exp(-C 2^{m+1})$ there are at least $2^{l+1}C$ active infections by the end of epoch $l + 1$. Therefore, for any $M \in \NN$, with probability at least $1 - \sum_{m=1}^M \exp(-B2^{m+1})$, at the end of the $M$-th epoch there are at least $2^MC$ active infections.
\end{proof}

To support the above lemma, we need to show that, so long as every active infected node generates a child on each step, the number of active infections continues to grow exponentially after tracing has begun.
\begin{lemma} \label{lem:bound_tracing}
Let $k = 3$. Let $X_t$ be the number of active infected nodes present at the end of step $t$. Suppose that the root is infected, and that at every step $t$ every active infected node generates an infected child. Then for all $t \geq 3$, $X_t \geq 2^t + 2$. 
\end{lemma}
\begin{proof}
Proof by induction. Consider the base case, $t = 3$. Let $T_t$ be the infection tree at time $t$. Recall that the infection process runs uninhibited until tracing begins in step $t = k$. Therefore, under these assumptions $T_2$ contains exactly $4$ active infected nodes. Also, recall that tracing proceeds first, followed by a round of the infection process. Therefore, $T_2$ describes the states of the infection tree at the start of step $t = 3$. The only node in the fronteir is the root, so the tracer stabilizes the root, and each of the remaining $3$ nodes generates a new infected child. Therefore $X_3 = (4-1) \cdot 2 = 6$. 

Assume that for all $n \geq 3$, $X_n \geq 2^n + 2$. Now consider $t = n + 1$. By assumption, at the start of the round there are $2^n + 2$ active infections. 
The tracer stabilizes one node, and each of the remaining nodes generates a new infected child. Therefore there are $X_{n+1} = (2^n + 2 - 1) \cdot 2 = 2^{n+1} + 2$ active infections at the end of round $t = n + 1$. Therefore, for all $t \geq 3$, $X_t \geq 2^t + 2$.
\end{proof}

\begin{theorem} [Paraphrased from Theorem 4.5 in~\cite{mitzenmacher_upfal}] \label{chernoff}
Let $Y_1, \dots, Y_n$ be independent Poisson trials such that $\Pr(Y_i) = u_i$. Let $Y = \sum_{i=1}^n Y_i$ and $\mu = \Exp[Y]$. Then for any $0 < \gamma \leq 1$, 
\begin{align*}
\Pr(Y \leq (1 - \gamma)\mu) \leq \exp(-\mu \gamma^2 / 2)
\end{align*}
\end{theorem}

\section{Deferred proofs from~\cref{section:arrival}} 
Recall that we need to decide how many trials to run in the second round of experiments, after observing the outcomes of the first set of trials. We can model the two policies as two coins $a$ and $b$, with unknown biases $p_a$ and $p_b$, respectively. We flip each coin $N$ times and observe heads $X_a = \tilde{p}_a N$ heads from coin $a$ and $X_b = \tilde{p}_b N$ heads from coin $b$. Without loss of generality, suppose that $\tilde{p}_a > \tilde{p}_b$, and let $d = \lvert \tilde{p}_a - \tilde{p}_b \rvert > 0$. What is the probability that $p_a > p_b$? 

\begin{lemma} \label{lem:2coin}
Let $\eps = .49 d$. If $\eps / p_a, \eps / p_b \in (0, 1)$,
\begin{align} \label{2coin_bd}
\Pr\left(p_a > p_b\right) &\geq 1 - 2\exp(- N \eps^2 / 3).
\end{align}
\end{lemma}

\begin{proof}
Let $E_a$ be the event that $p_a > \tilde{p}_a - \eps$, and let $E_b$ be the event that $p_b < \tilde{p}_b + \eps$. If $E_a$ and $E_b$ both hold, then 
$$p_a > \tilde{p}_a - d/2 = \tilde{p}_b + d/2 > p_b.$$ 

Therefore,
\begin{align*}
\Pr\left(p_a > p_b\right) &\geq \Pr\left(E_a \land E_b\right) \\
&= 1 - \Pr\left(\lnot E_a \lor \lnot E_b\right) \\
&\geq 1 - (\Pr\left(\lnot E_a\right) + \Pr\left(\lnot E_b\right))
\end{align*}

Most of the remaining work follows the analysis of section 4.2.3 ``Application: Estimating a Parameter'' in~\cite{mitzenmacher_upfal}. Let $\mu_a = p_a N$ and let $\mu_b = p_b N$.
\begin{align*}
\Pr\left(\lnot E_a\right) + \Pr\left(\lnot E_b\right) &= \Pr\left(p_a \leq \tilde{p}_a - \eps\right) + \Pr\left(p_b \geq \tilde{p}_b + \eps\right) \\
&= \Pr\left(p_a N \leq \tilde{p}_a N - \eps N\right) + \Pr\left(p_b N \geq \tilde{p}_b N + \eps N\right) \\
&= \Pr\left(p_a N \leq X_a - \eps N\right) + \Pr\left(p_b N \geq X_b + \eps N\right) \\
&= \Pr\left(p_a N + \eps N \leq X_a \right) + \Pr\left(p_b N - \eps N \geq X_b \right) \\
&= \Pr\left(X_a \geq p_a N + \eps N \right) + \Pr\left(X_b \leq p_b N - \eps N \right) \\
&= \Pr\left(X_a \geq p_a N + (\eps / p_a) p_a N \right) + \Pr\left(X_b \leq p_b N - (\eps / p_b) p_b N \right) \\
&= \Pr\left(X_a \geq (1 + (\eps / p_a)) p_a N \right) + \Pr\left(X_b \leq (1 - (\eps / p_b)) p_b N \right) \\
&= \Pr\left(X_a \geq (1 + (\eps / p_a)) \mu_a \right) + \Pr\left(X_b \leq (1 - (\eps / p_b)) \mu_b\right)
\end{align*}
If $\eps / p_a, \eps / p_b \in (0, 1)$, then we can apply~\cref{bound_big} to the first term and~\cref{bound_small} to the second term. Continuing with the above calculations, we have that
\begin{align*}
\Pr\left(\lnot E_a\right) + \Pr\left(\lnot E_b\right) &= \Pr\left(X_a \geq (1 + (\eps / p_a)) \mu_a \right) + \Pr\left(X_b \leq (1 - (\eps / p_b)) \mu_b\right) \\
&\leq \exp(-\mu_a (\eps / p_a)^2 / 3) + \exp(-\mu_b (\eps / p_b)^2 / 2) \\
&= \exp(-p_a N (\eps / p_a)^2 / 3) + \exp(-p_b N (\eps / p_b)^2 / 2) \\
&= \exp(- N (\eps^2 / p_a) / 3) + \exp(- N (\eps^2 / p_b) / 2) \\
&\leq \exp(- N \eps^2 / 3) + \exp(- N \eps^2 / 2) \\
&\leq 2\exp(- N \eps^2 / 3)
\end{align*}
The penultimate inequality comes from the fact that $p_a, p_b \leq 1$. Therefore
\begin{align*}
\Pr\left(p_a > p_b\right) &\geq 1 - (\Pr\left(\lnot E_a\right) + \Pr\left(\lnot E_b\right)) \\
&\geq 1 - 2\exp(- N \eps^2 / 3)
\end{align*}
\end{proof}

\subsection{Implementation details for applying the above bound in~\cref{section:arrival}}
How can we verify that $\eps / p_a, \eps / p_b \in (0, 1)$, if $p_a$ and $p_b$ are unknown? While we don't know $p_a$ or $p_b$, we \textit{do} know that $p_a, p_b \geq 1 - p_{\textrm{infection}}$. This follows from the fact that the root is infected with probability $p_{\textrm{infection}}$, and if the root node is not infected, the infection is trivially contained. 

So for a fixed confidence threshold $m \in [0, 1]$, to determine whether $p_a > p_b$ with confidence at least $m$, we need to implement the following process in the code:
\begin{enumerate}
\item Compute $p_0 = 1 - p_{\textrm{infection}}$.
\item Compute $\eps = .49 \lvert \tilde{p}_a - \tilde{p}_b \rvert$.
\item If $\eps / p_0 > 1$, report ``no confidence'' for that pixel. 
\item Otherwise, $\eps / p_a \leq \eps / p_0 < 1$ and $\eps / p_b \leq \eps / p_0 < 1$, so our Chernoff bound applies. 
\item Compute~\cref{2coin_bd}. If the confidence is at least $m$, report that $p_a > p_b$ with confidence at least $m$; else report ``no confidence''.
\end{enumerate}

\section{Deferred proofs from~\cref{section:beyond}}
We can model the three policies as threes coins $a$, $b$, and $c$, with unknown biases $p_a$, $p_b$, and $p_c$, respectively. We flip each coin $N$ times and observe heads $X_a = \tilde{p}_a N$ heads from coin $a$, $X_b = \tilde{p}_b N$ heads from coin $b$, and $X_c = \tilde{p}_c N$ heads from coin $c$. Without loss of generality, suppose that $\tilde{p}_a > \tilde{p}_b > \tilde{p}_c$. Let $d_1 = \lvert \tilde{p}_a - \tilde{p}_b \rvert$ and let $d_2 = \lvert \tilde{p}_a - \tilde{p}_c \rvert$. What is the probability that $p_a = \max(p_a, p_b, p_c)$? 

\begin{lemma} \label{lem:3coin_bd}
Let $\eps_1 = .49 d_1$ and let $\eps_2 = .49 d_2$. If $\eps_1 / p_a, \eps_1 / p_b, \eps_2 / p_c \in (0, 1)$, then 
\begin{align} \label{3coin_bd}
\Pr(p_a = \max(p_a, p_b, p_c)) &\geq 1 - 3 \exp(- N \eps_1^2 / 3)
\end{align}
\end{lemma}

\begin{proof}
Let $E_A$ be the event that $p_a > \tilde{p}_a - \eps_1$, let $E_B$ be the event that $p_b < \tilde{p}_b + \eps_1$, and let $E_C$ be the event that $p_c < \tilde{p}_c + \eps_2$. Note that $\eps_1 \leq \eps_2$. If $E_A$ and $E_B$ both hold, then 
\begin{align*}
p_a \geq \tilde{p}_a - \eps_1 > \tilde{p}_b + \eps_1 \geq p_b
\end{align*}
And if $E_A$ and $E_C$ both hold, then
\begin{align*}
p_a \geq \tilde{p}_a - \eps_1 \geq \tilde{p}_a - \eps_2 > \tilde{p}_c + \eps_2 \geq p_c
\end{align*}
Therefore, if $E_A$, $E_B$, and $E_C$ all hold simultaneously, $p_a = \max(p_a, p_b, p_c)$. We have that
\begin{align*}
\Pr(p_a = \max(p_a, p_b, p_c)) &\geq \Pr(E_A \land E_B \land E_C) \\
&= 1 - \Pr(\lnot E_A \lor \lnot E_B \lor \lnot E_C) \\
&\geq 1 - (\Pr(\lnot E_A) + \Pr(\lnot E_B) + \Pr(\lnot E_C))
\end{align*}

First, let's bound the right hand side. Let $\mu_a = p_a N$, $\mu_b = p_b N$, and $\mu_c = p_c N$. 
\begin{align*}
\Pr(\lnot E_A) + \Pr(\lnot E_B) + \Pr(\lnot E_C) &= \Pr(p_a \leq \tilde{p}_a - \eps_1) + \Pr(p_b \geq \tilde{p}_b + \eps_1) + \Pr(p_c \geq \tilde{p}_c + \eps_2) \\
&\leq \Pr(p_a N \leq \tilde{p}_a N - \eps_1 N) + \Pr(p_b N \geq \tilde{p}_b N + \eps_1 N) + \Pr(p_c N \geq \tilde{p}_c N + \eps_2 N) \\
&\leq \Pr(p_a N + \eps_1 N \leq \tilde{p}_a N ) + \Pr(p_b N - \eps_1 N \geq \tilde{p}_b N) + \Pr(p_c N - \eps_2 N \geq \tilde{p}_c N) \\
&\leq \Pr(\tilde{p}_a N \geq p_a N + \eps_1 N) + \Pr(\tilde{p}_b N \leq p_b N - \eps_1 N) + \Pr(\tilde{p}_c N \leq p_c N - \eps_2 N) \\
&\leq \Pr(X_a \geq (1 + (\eps_1 / p_a)) \mu_a) + \Pr(X_b \leq (1 - (\eps_1 / p_b)) \mu_b) + \Pr(X_c \leq (1 - (\eps_2 / p_c)) \mu_c)
\end{align*}
If $\eps_1 / p_a, \eps_1 / p_b, \eps_2 / p_c \in (0, 1)$, then we can apply~\cref{bound_big} to the first term and~\cref{bound_small} to the second two terms. Continuing with the above calculations, we have that
\begin{align*}
\Pr(\lnot E_A) + \Pr(\lnot E_B) + \Pr(\lnot E_C) &\leq \exp(-\mu_a (\eps_1 / p_a)^2 / 3) + \exp(-\mu_b (\eps_1 / p_b)^2 / 2) + \exp(-\mu_c (\eps_2 / p_c)^2 / 2) \\
&\leq \exp(- N \eps_1^2 /(3 p_a)) + \exp(- N \eps_1^2 / (2 p_b)) + \exp(- N \eps_2^2 / (2 p_c)) \\
&\leq \exp(- N \eps_1^2 / 3) + \exp(- N \eps_1^2 / 2) + \exp(- N \eps_2^2 / 2) \\
&\leq 3 \exp(- N \eps_1^2 / 3)
\end{align*}
The penultimate inequality comes from the fact that $p_a, p_b, p_c \leq 1$. The final inequality comes from the fact that $\eps_2 > \eps_1$. Therefore,
\begin{align*}
\Pr(p_a = \max(p_a, p_b, p_c)) &\geq 1 - 3 \exp(- N \eps_1^2 / 3)
\end{align*}
\end{proof}

\subsection{Implementation details for applying the above bound in~\cref{section:beyond}}
In the experiment in~\cref{section:beyond}, the probability of infection is drawn from the distribution $\Unif(p_{\min}, 1)$. Therefore, for a fixed $p_{\min}$, the probability that the root is infected is $p_{\textrm{infected}} = (p_{\min} + 1) / 2$. Therefore, $p_a, p_b, p_c \geq p_0$, where  
\begin{align*}
p_0 &\geq 1 - p_{\textrm{infected}} \\
&= 1 - (p_{\min} + 1) / 2,
\end{align*}

The process is similar as before. For a fixed confidence threshold $m \in [0, 1]$, to determine whether $p_a = \max(p_a, p_b, p_c)$ with confidence at least $m$, we need to implement the following process in the code:
\begin{enumerate}
\item Compute $p_{\textrm{infection}} = (1 + p_{\min}) / 2$ for each pixel.
\item Compute $p_0 = 1 - p_{\textrm{infection}}$ for each pixel.
\item Compute $\eps_1 = .49 \lvert \tilde{p}_a - \tilde{p}_b \rvert$ and $\eps_2 = .49 \lvert \tilde{p}_a - \tilde{p}_c \rvert$.
\item If $\eps_1 / p_0 > 1$ or $\eps_2 / p_0 > 1$, report ``no confidence'' for that pixel. 
\item Otherwise, $\eps_1 / p_a \leq \eps_1 / p_0 < 1$, $\eps_1 / p_b \leq \eps_1 / p_0 < 1$, and $\eps_2 / p_c \leq \eps_2 / p_0 < 1$, so our Chernoff bound applies. 
\item Compute~\cref{3coin_bd}. If the confidence is at least $m$, report that $p_a = \max(p_a, p_b, p_c)$ with confidence at least $m$; else report ``no confidence''.
\end{enumerate}

\subsection{Referenced Chernoff Bounds}
\begin{theorem} \label{bound_big} [Thm 4.4 (2) in~\cite{mitzenmacher_upfal}, almost verbatim] Let $X_1, \dots, X_n$ be independent Poisson trials such that $\Pr(X_i) = p_i$. Let $X = \sum_{i=1}^n X_i$ and $\mu = \Exp\left(X\right)$. Then the following Chernoff bound holds: for $0 < \delta \leq 1$, 
$$\Pr(X \geq (1 + \delta)\mu) \leq \exp(-\mu \delta^2 / 3).$$
\end{theorem}

\begin{theorem} \label{bound_small} [Thm 4.5 (2) in~\cite{mitzenmacher_upfal}, almost verbatim] Let $X_1, \dots, X_n$ be independent Poisson trials such that $\Pr(X_i) = p_i$. Let $X = \sum_{i=1}^n X_i$ and $\mu = \Exp\left(X\right)$. Then, for $0 < \delta < 1$:
$$\Pr(X \leq (1 - \delta)\mu) \leq \exp(-\mu \delta^2 / 2).$$
\end{theorem}

\end{document}